\documentclass[a4paper,11pt]{amsart}

\usepackage{amssymb}
\usepackage{amsthm}
\usepackage{graphics}
\usepackage{amsmath}
\usepackage{amstext}

\usepackage{fancyhdr}

\usepackage{xfrac}
\usepackage{braket}
\usepackage{mathtools}
\usepackage{stmaryrd}
\usepackage{mathrsfs}
\usepackage{extarrows}

\usepackage{microtype}

\usepackage{graphicx}
\usepackage{caption}
\usepackage{subcaption}

\usepackage{todonotes}
\usepackage{import}

\usepackage{ esint }
\usepackage{abstract}

\usepackage{tikz}
\usetikzlibrary{matrix,arrows,calc,patterns}

\title{Parabolic implosion for endomorphisms of $\C^2$}

\author{Fabrizio Bianchi}

\thanks{This research is partially supported by the ANR project LAMBDA, ANR-13-BS01-0002 and by the FIRB2012 grant ``Differential Geometry and Geometric Function Theory", RBFR12W1AQ 002.}

\input pdfcolor

\newcommand{\Bb}{{\mathcal B}}
\newcommand{\Cc}{{\mathcal C}}

\newcommand{\Rr}{{\mathcal R}}

\newcommand{\Tt}{{\mathcal T}}
\newcommand{\Uu}{{\mathcal U}}

\def\eps{\varepsilon}

\newcommand{\mb}{\mathbb}

\newcommand{\R}{\mb R}
\newcommand{\C}{\mb C}
\newcommand{\N}{\mb N}
\newcommand{\D}{\mb D}

\renewcommand{\P}{\mb P}

\renewcommand{\phi}{\varphi}
\renewcommand{\epsilon}{\varepsilon}
\renewcommand{\bar}{\overline}
\renewcommand{\tilde}{\widetilde}

\def\({\left(}
\def\){\right)}

\DeclareMathOperator{\realpart}{Re}
\renewcommand{\Re}{\realpart}
\DeclareMathOperator{\imagpart}{Im}
\renewcommand{\Im}{\imagpart}

\def\b#1{\bar{#1}}
\def\t#1{\tilde{#1}}

\newcommand{\pa}[1]{\left(#1\right)}

\renewcommand{\bra}[1]{\left[#1\right]}
\newcommand{\abs}[1]{\left|#1\right|}

\newtheorem{teo*}{Theorem}

\newtheorem{teo}{Theorem}[section]
\newtheorem{defi}[teo]{Definition}
\newtheorem{cor}[teo]{Corollary}
\newtheorem{lemma}[teo]{Lemma}
\newtheorem{prop}[teo]{Proposition}
\newtheorem{remark}[teo]{Remark}

\newcommand{\fai}{\phi^\iota}
\newcommand{\fao}{\phi^o}

\newcommand{\aeps}{\abs{\eps}}

\begin{document}

\maketitle

\begin{abstract}
We give an estimate of the discontinuity of the large Julia set
for a perturbation of a class of maps tangent to the identity,
by means of a two-dimensional
Lavaurs Theorem.
We adapt to our situation a strategy due to Bedford, Smillie and Ueda in the semiattracting setting.
We also prove the discontinuity of the filled Julia set for such perturbations of
regular polynomials.
\end{abstract}

\section{Introduction and results}

Parabolic dynamics and the study of parabolic perturbations
have been at the heart of holomorphic dynamics
in the last couple of decades. 
Starting with the work by Lavaurs \cite{lavaurs},
the theory of parabolic implosion has provided
useful tools
for a very precise
control of
these perturbations and for the proof
of some of the most striking recent results in the field,
e.g.,
the construction of Julia sets of positive area \cite{buff2008quadratic,avila2015lebesgue},
significative steps toward the setting of the \emph{hyperbolicity conjecture}
for quadratic polynomials \cite{cheraghi2015satellite} and the construction
of endomorphisms of $\P^2 (\C)$ with a wandering domain \cite{astorg2016two}.
In particular, these techniques have proved extremely useful in
the study of
bifurcation
loci
(see, e.g., \cite{shishikura1998hausdorff}).

In several complex variables, the study
of parabolic perturbations and
the theory of parabolic implosion
are just at the start,
with
recent results only
in the semiattracting setting \cite{bsu, dujardin2013stability}.
The goal of this paper is to provide a starting point for an analogous theory
in the completely parabolic setting,
by a precise study of perturbations of germs of endomorphisms of $\C^2$
tangent to the identity at the origin.

Let us briefly recall the foundational results
of the one-dimensional theory. We refer to \cite{douady94}
for a more extended introduction to the subject, as well as to the original work by Lavaurs \cite{lavaurs}.
Consider the polynomial map on $\C$ tangent to the identity given by
$f(z)=z+z^2$. The origin is a \emph{parabolic}
fixed point for $f$. The dynamics is attracting near the negative real axis:
there exists a parabolic basin $\Bb$ for 0, i.e., an open set of points
converging to the origin after iteration. The origin is on the boundary of $\Bb$, and the convergence
happens tangentially to
the negative real axis. The iteration of $f$ on $\Bb$ is semiconjugated to a translation by 1. More precisely,
there 
exists an \emph{incoming Fatou coordinate} $\fai:\Bb\to \C$ such that, for every $z\in \Bb$, we have
$\fai \circ f (z)=f(z)+1 $.

The same happens for the inverse iteration near the positive real axis: we have a repelling
basin $\Rr$ of points converging to 0 under some
inverse iteration, and the convergence happens tangentially to
the positive real axis. We can construct in this case an \emph{outgoing Fatou parametrization},
i.e., a map $\psi^o:\C \to \Rr$ such that
$f\circ \psi^o (z) = \psi^o (z+1)$.
It is worth noticing here that
the union of
$\Bb$ and $\Rr$
gives a full pointed neighbourhood of the origin.

Notice that the
incoming Fatou coordinate is a map from the dynamical plane to $\C$, while
the outgoing Fatou parametrization
is a map from $\C$ to the dynamical plane.
In particular, given any $\alpha\in \C$
and denoting by $t_\alpha$
the translation by $\alpha$
on $\C$,
the composition $L_\alpha :=\psi^o \circ t_\alpha \circ \fai$
is well defined as a function from $\Bb$ to $\Rr$. Such a map is usually
called a \emph{Lavaurs map}, or a
\emph{transfer map}.

We consider now the perturbation $f_\eps (z)=z+z^2 + \eps^2$
of the system $f$, for $\eps$ real and positive. As $\eps\neq 0$,
the dynamics abruptly changes: the parabolic point splits in two
(repelling) points $\pm i\eps$, and the orbits of points in $\Bb$ can now pass through the ``gate"
between these two points.
Using the Lavaurs map it is possible
to give a very precise description of this phenomenon,
by studying
the dynamics of high iterates of the perturbed maps $f_\eps$, as $\eps\to 0$.
The following definition plays a central role
in this study.

\begin{defi}\label{defi_alpha_sequence}
Given $\alpha \in \C$,
an \emph{$\alpha$-sequence} is a sequence
$(\eps_\nu, n_\nu)_{\nu \in \N} \in \pa{\C\times \N}^\N$
such that $n_\nu \to \infty$
and
$n_\nu - \frac{\pi}{\eps_\nu}\to \alpha$ as $\nu\to\infty$.
\end{defi}
Notice in particular that the definition of $\alpha$-sequence
implies that $\eps_\nu$ tends to the origin
tangentially to the positive real axis. More precisely, there exists
a constant $c$ such that, for every $\nu$ sufficiently large,
we have
$\abs{\Im \eps_\nu} \leq c \abs{\eps_\nu}^2$.
The following result gives the limit description of suitable high iterates of
$f_\eps$.

\begin{teo}[Lavaurs \cite{lavaurs}]\label{teo_lavaurs_convergence}
Let $f_\eps (z)= z+z^2 + \eps^2 + o(z^2, \eps^2)$ and $(\eps_\nu, n_\nu)$
be an $\alpha$-sequence.
Then $f^{n_\nu}_{\eps_\nu} \to L_\alpha$, locally uniformly on $\Bb$.
\end{teo}

One of the most direct consequences of Lavaurs theorem
is the fact that the set-valued functions
$\eps\mapsto J(f_\eps)$
and $\eps \mapsto K(f_\eps)$ are
discontinuous for the Hausdorff topology
at $\eps =0$.
 Here
$J(f_\eps)$
 and $K(f_\eps)$
denote the Julia set
 and the filled Julia set
of $f_\eps$, respectively (recall -- see e.g. \cite{douady94} -- that $J(f_\eps)$
is always lower semicontinuous, while $K(f_\eps)$
is always upper semicontinuous).
More precisely, define the \emph{Lavaurs-Julia set} $J (f_0, L_\alpha)$
and the \emph{filled Lavaurs-Julia set} $K(f_0, L_\alpha)$
by
\[
\begin{aligned}
& J (f_0,L_\alpha):= \b{ \set{z\in \C \colon \exists m \in \N, L^m_\alpha (z) \in J (f_0)}}\\
& K(f_0, L_\alpha):= \set{z\in \C \colon \exists m\in \N, L^m_\alpha (z)\notin K(f_0) }^c
\end{aligned}
\]
Notice that the Lavaurs-Julia set $J(f_0, L_\alpha)$
is in general larger than the Julia set of $f_0$.
On the other hand, the set $K(f_0, L_\alpha)$
is in general smaller than $K(f_0)$.
The following Theorem
then gives an estimate of the discontinuity of the maps $\eps \mapsto J(f_\eps)$ at $\eps=0$.

\begin{teo}[Lavaurs \cite{lavaurs}]\label{teo_lavaurs_disc_JK}
Let $f_\eps (z)= z+z^2 + \eps^2 + o(z^2+\eps^2)$ and $(\eps_\nu, n_\nu)$ be an $\alpha$-sequence.
Then
\[
\liminf J (f_{\eps_\nu}) \supset J(f_0, L_\alpha)
\mbox{ and }
\limsup K(f_{\eps_\nu}) \subset K(f_0, L_\alpha)
\]
In particular, at $\eps=0$,
\begin{enumerate}
\item the map $\eps \to J(f_\eps)$ is lower semicontinuous, but not continuous;
\item the map $\eps \to K(f_\eps)$ is upper semicontinuous, but not continuous.
\end{enumerate}
\end{teo}

The goal of this paper is to make
a step toward the generalization of
Theorems \ref{teo_lavaurs_convergence}
and \ref{teo_lavaurs_disc_JK}
to the two-variable setting, by studying the perturbation of
a class of maps tangent to the identity (i.e., with differential at a fixed point equal to the identity).
More precisely, we
consider
an endomorphism of $\C^2$ of the form 
(see at the end of the introduction for the notation used)
\begin{equation}\label{eq_F0}
F_0
\left(
\begin{array}{c}
 x\\
 y
\end{array}\right)
=
\left(\begin{array}{c}
       x+ x^2 (1+ (q+1)x + ry +O(x^2, xy,y^2))\\
       y (1+ \rho x + O(x^2,xy,y^2))
      \end{array}\right),
\end{equation}
where $\rho$ is real and greater than 1 and $q,r\in \C$.
For instance, $F_0$ may be the local expression of
an endomorphism of $\P^2$
(e.g., if the two components of $F_0$
are polynomials of the same degree in $(x,y)$
with 0 as the only common root
of their higher-degree
homogeneous parts).
We shall primarily be interested in this situation.

The map $F_0$ has a fixed point
tangent to the identity at the origin, and two invariant lines
$\set{x=0}$ and $\set{y=0}$.
By the work of Hakim \cite{hakim97}
(see
Section \ref{section_prelim}) we know that
$[1:0]$ is a \emph{non-degenerate
characteristic direction},
and that there exists an open set $\Bb$
of initial conditions, with the origin on the boundary, such that
every point in $\Bb$
is attracted to the origin tangentially to the direction $[1:0]$.
Moreover there exists, on an open subset $\t C_0$ of $\Bb$,
a (one dimensional)
Fatou coordinate $\t \fai$, with values in $\C$, such that
$\t  \fai \circ F_0 (p) =\t \fai(p)+1$ (see Lemma  \ref{lemma_prima_coord}). 

A similar description holds
for the inverse map.
Indeed, after restricting ourselves to a neighbourhood $U$
of the origin where $F_0$
is invertible,
we can define the set $\Rr$ of point that are attracted to the origin tangentially to the
direction $[1:0]$
by backward iteration. There is then a well defined map
$\t \fao: -\t C_0 \cap U \to \C$
such that $\t \fao \circ F_0 (p) = \t \fao (p)+1$ whenever the left hand side is defined.
It is actually possible to construct two-dimensional Fatou coordinates (see \cite{hakim97}),
but we shall not need them in
this work.

Consider now a perturbation $F_\eps$ of $F_0$
of the form
\begin{equation}\label{eq_family}
\begin{aligned}
F_\eps
\left(
\begin{array}{c}
 x\\
 y
\end{array}\right)
&=
\left(
\begin{array}{c}
 x+(x^2 +\eps^2) \alpha_\eps (x,y)\\
       y (1+ \rho x + \beta_\eps (x,y))
\end{array}
\right)
 \\
& =
\left(\begin{array}{c}
       x+(x^2 +\eps^2)(1+ (q+1)x +ry + O(x^2,xy,y^2) + O(\eps^2))\\
       y (1 + \rho x + O(x^2,xy,y^2) + O(\eps^2) )
      \end{array}\right).
      \end{aligned}
\end{equation}
Our goal is to study the dependence of the large Julia
set\footnote{i.e., the complement of the Fatou set, which in general is larger
than the Julia set
defined as the support of the equilibrium measure for endomorphisms of $\P^2$, see \cite{ds_cime}.}
$J^1 (F_\eps)$
on $\eps$ near the parameter $\eps=0$.
Our main result is the following Theorem, which is a partial
generalization of Theorem \ref{teo_lavaurs_convergence}
to our setting.
As in dimension 1, \emph{$\alpha$-sequences} play a crucial role. The set $\t C_0$
introduced above will be precisely defined in Proposition \ref{prop_hakim_26}, and the Fatou coordinates $\t \fai$
and $\t \fao$ in Lemma \ref{lemma_prima_coord}.

\begin{teo}\label{teo_lav_2d}
 Let
 $F_\eps$ be a holomorphic family of endomorphisms of $\C^2$ as in \eqref{eq_family}.
 Let $F_0$ be invertible on a neighbourhood $U$ of the origin
 and let $\t \fai: \t C_0 \to \C$ and
 $\t \fao: - \t C_0 \to \C$
 be the (1-dimensional) Fatou coordinates for $F_0$.
 Let $\Bb$ be the attracting basin
 for the origin for the map $F_0$ with respect to the characteristic direction $[1:0]$ and $\Rr$
 the repelling one.
 Let $\alpha$ be a complex number
 and $(n_\nu, \eps_\nu)$
 be an $\alpha$-sequence.
 Then
 every
compact subset $\Cc \subset \Bb \cap \set{y=0}$ 
 has a neighbourhood $U_\Cc$
 where,
 up to extracting a subsequence, we have
 \[
 F^{n_\nu}_{\eps_\nu} \to T_\alpha
 \]
 locally uniformly,
 where $T_\alpha$ is a well defined open holomorphic map from $U_{\Cc}$ to $\C^2$, with values in $\Rr$.
 Moreover,
 \begin{equation}\label{eq_teo_lav_2d_semiconj}
 \t \fao \circ T_\alpha (p) = \alpha + \t \fai (p)
 \end{equation}
whenever
both sides are defined.
\end{teo}

As a consequence, we shall deduce an estimate
of the
discontinuity
of the large Julia set in this context
(notice that the discontinuity itself
follows from an application of Theorem
\ref{teo_lavaurs_convergence}
to the invariant line $\set{y=0}$).
We say that, given $\Uu\subset \t C_0$, a map $T_\alpha: \Uu \to \C^2$
is a \emph{Lavaurs map} if there exists an $\alpha$-sequence $(\eps_\nu, n_\nu)$
such that
$F_{\eps_\nu}^{n_\nu}\to T_\alpha$
on $\Uu$.
We then have the following result
(see Section \ref{section_disc_J} for the definition
of the Lavaurs-Julia sets $J^1(F_0, T_\alpha)$
in this setting).

\begin{teo}\label{teo_disc_J}
Let $F_\eps$ be a holomorphic family of endomorphisms of $\P^2$ as in \eqref{eq_family}
and $T_\alpha:\Uu\to\C^2$ be a Lavaurs map such that $F_{\eps_\nu}^{n_\nu}\to T_\alpha$
on $\Uu$ for some $\alpha$-sequence $(\eps_\nu, n_\nu)$.
Then
\[
 \liminf J^1 (F_{\eps_\nu}) \supset J^1 (F_0, T_\alpha).
 \]
\end{teo}

Finally, in the last section, we consider a family of \emph{regular polynomials}, i.e., polynomial endomorphisms
of $\C^2$ admitting an extension to $\P^2 (\C)$. For these maps, it is meaningful to define the \emph{filled
Julia set} $K$ as the set of points with bounded orbit. In analogy with the one-dimensional theory, we deduce from Theorem \ref{teo_lav_2d}
an estimate
for the discontinuity of the filled Julia set at $\eps=0$ (see Section \ref{section_altro} for
the definition of the set $K(F_0, T_\alpha)$) and in particular deduce that $\eps\mapsto K(F_\eps)$
is discontinuos at $\eps=0$.
Notice that, differently from the case of the large Julia set, this is not already a direct consequence
of the 1-dimensional theory.

\begin{teo}\label{teo_K_intro}
Let $F_\eps$ be a holomorphic family of regular polynomial maps of $\C^2$
as in \eqref{eq_family}
and $T_\alpha:\Uu\to\C^2$ be a Lavaurs map such that $F_{\eps_\nu}^{n_\nu}\to T_\alpha$
on $\Uu$ for some $\alpha$-sequence $(\eps_\nu, n_\nu)$.
Then
 \[
K (F_0, T_\alpha) \supset \limsup
K (F_{\eps_j}).
 \]
 Moreover,
 $\eps\mapsto K(F_\eps)$ is discontinuous at $\eps=0$.
\end{teo}

The paper is organized as follows. In Section \ref{section_prelim}
we recall the results by Hakim
describing the local dynamics
of the map \eqref{eq_F0}
near the origin, and introduce the Fatou coordinates associated to the attracting and repelling basins.
In Section \ref{section_perturbed} we define and study suitable perturbations
of the Fatou coordinates, that allow to semiconjugate the iteration of $F_\eps$ to a translation by 1, up
to a controlled error.
In Sections \ref{section_fasi} and \ref{section_prelim_convergence}
we carefully study the orbits of points under iteration by $F_\eps$ and prove some preliminary convergence
result needed for the proof of Theorem \ref{teo_lav_2d}, which is given in Section \ref{section_convergence}.
In Section \ref{section_disc_J} and \ref{section_altro}
we deduce from Theorem \ref{teo_lav_2d} the estimates of the discontinuity of the large Julia set
and (for regular polynomials) of the filled Julia set 
at $\eps =0$.

\subsection*{Notation}

The symbol $O(x)$ will stand for some element in the ideal generated by $x$. More
generally, given any $f$, $O(f)$ will stand 
for
some element in the ideal generated by $f$. Analogously,
$O(f_1,\dots,f_k)$ will stand for some element in the ideal generated by $f_1, \dots, f_k$.
 
The notation $O_2 (x,y)$ will be a shortcut for $O(x^2, xy, y^2)$.
Given a point $p\in \C^2$, we shall denote its components as $x(p)$ and $y(p)$.

\subsection*{Acknowledgements}
It is a pleasure to thank my advisor Marco Abate for suggesting me this problem,
for his many helpful comments and the careful reading of this paper.
I would also like to thank Eric Bedford for helpful discussions about his paper.

\section{Preliminaries and Fatou coordinates}\label{section_prelim}

Following the work of Hakim \cite{hakim97} (see also \cite{hakim98,arizzi2014ecalle}),
we start giving a description of the local
dynamics
near the origin for $F_0$
by recalling some classical notions in this setting.
Let $\Phi$ be a germ of transformation tangent to the identity at the origin
of $\C^2$. We can locally write it near the origin as
\[
\Phi \pa{
\begin{array}{c}
 x\\
 y
\end{array}
}
=
\pa{
\begin{array}{c}
 x + P(x,y) + \dots\\
 y + Q(x,y) + \dots
\end{array}
},
\]
where $P$ and $Q$ are homogeneous polynomials of degree $2$.
In the following, we shall always assume that $P(x,y)$
is not identically zero.
A \emph{characteristic direction}
is a direction $V = [x:y]\in \P^1 (\C)$ such that the complex line through the origin in the direction $[x:y]$ is invariant
for $(P,Q)$. The direction is \emph{degenerate} if the restriction of $(P,Q)$ is zero on it, \emph{non degenerate} otherwise.

Consider now a non degenerate characteristic direction $V$ and take coordinates such that $V = [1:u_0]$.
Notice that the fact that $[1:u_0]$ is a characteristic direction is equivalent to $u_0$ being a zero of
$r(u) := Q(1,u) - u P(1,u)$.
The \emph{director} of the characteristic direction $[1:u_0]$ is thus defined as
\[
\frac{r'(u_0)}{P(1,u_0)}
\]
(see \cite[Definition 2.4]{abate2015fatou} for a more intrinsec -- and equivalent -- definition).
Given a germ $\Phi$ and a non degenerate characteristic direction $V$ for $\Phi$
we can assume, without loss of generality,
that $V=[1:0]$
and that the coefficient
of $x^2$ in $P(x,y)$
is 1 (notice that Hakim has the opposite normalization, i.e., with the term $-x^2$).
The following result by Hakim
(\cite[Proposition 2.6]{hakim97})
gives an explicit description of an invariant subdomain of $\Bb$.
In all this work, we will restrict ourselves to points belonging to such an invariant domain.

\begin{prop}[Hakim]\label{prop_hakim_26}
 Let $\Phi$ be a germ of transformation of $\C^2$ tangent to the identity (normalized as above), such that $V = [1:0]$
 is a nondegenerate characteristic direction with director $\delta$ whose real part is
 greater than some $0< \alpha\in \R$. Then,
if $\gamma, s$ and $R$ are small enough positive constants,
 every point of the set
 \[
\tilde C_0 (\gamma, R,s):=
  \set{(x,y)\in \C^2 \colon \abs{\Im x}\leq - \gamma  \Re x , \abs{x}\leq R, \abs{y}\leq s \abs{x}}
 \]
 is attracted to the origin in the direction $V$ and $x(\Phi^n (x,y)) \sim -\frac{1}{n}$.
 Moreover we have $\abs{x_n}\leq {\frac{2}{n}}$ and
 \begin{equation}\label{eq_xy_hakim}
 \abs{y(\Phi^n (x,y))} \abs{x(\Phi^n (x,y))}^{-\alpha-1} \leq \abs{y} \abs{x}^{-\alpha -1}.
 \end{equation}
\end{prop}

Notice that, for a $\gamma_1$ slightly smaller
than $\gamma$, we have $F_0 (\t C_0 (\gamma, R,s)) \subseteq \t C_0 (\gamma, R,s)$.

Let us now consider $F_0$ as in \eqref{eq_F0}. It is immediate to see that $[1:0]$
is a non-degenerate characteristic direction, with director equal to $\rho -1$. This is the reason
we made the assumption that $\rho > 1$.
It will be even clearer later (Lemma \ref{lemma_serie}) that this a crucial assumption.

An important feature of our setting is that the (local) inverse of a map tangent to the identity
shares a lot of properties with the original map (this does not happen for instance in the semi-parabolic situation).
In fact, it is immediate to see
that the local inverse of an endomorphism tangent to the identity is still
tangent to the identity, with the same characteristic directions and moreover the same Hakim directors.
In our situation, $(0,0)$ is still a double fixed point for the local inverse $G_0$,
which has the following form (see for example the explicit description of the coefficients of the inverse of
an endomorphism tangent to the identity given in \cite{arformal}),
\[G_0
\left(
\begin{array}{c}
 x\\
 y
\end{array}\right)
=
\left(\begin{array}{c}
       x - x^2 (1+   (q-1)x + r y + O_2(x,y))\\
       y(1 - \rho x + O_2(x,y))
      \end{array}\right)
\]
and the stated properties are readily verified.

In the following, we will fix a neighbourhood $U$
of the origin where $F_0$ is invertible, and consider an invariant domain $\t C_0$ as in
Proposition \ref{prop_hakim_26} for $F_0$ such that
$-\t C_0$ satisfies the same property for $G_0$ and 
both $\t C_0$ and $-\t C_0$
are contained in $U$.

We now briefly recall how to construct
a (one dimensional)
Fatou coordinate $\t \fai$ on $\t C_0$
semiconjugating $F_0$ to a translation
by 1.
We notice here that it is actually possible to construct a two-dimensional
Fatou coordinate, on a subset of $\t C_0$, with values in $\C^2$ and semiconjugating the system to the translation
by $(1,0)$. Since we will not use it, we do not detail the construction here, but we refer the interested
reader to \cite{hakim97}.

The first step of the construction of $\t \fai$
is to consider the map
\begin{equation}\label{eq_def_w0}
\t w^\iota_0  (x,y ):=  -\frac{1}{x} - q \log (-x).
\end{equation}
Notice that, in the chart $\t w^\iota_0$, the map $F_0$
already looks like a translation by 1. Indeed, 
by \eqref{eq_F0}, we have
\begin{equation}\label{eq_conto_w0}
\begin{aligned}
w^{\iota}_0 (F_0 (x,y)) & = - \frac{1}{x(F_0 (x,y))} - q \log (- x(F_0 (x,y)))\\
& = -\frac{1}{x} - q \log (-x) + 1 + ry + O_2(x,y)\\
& = \t w^\iota_0 (x,y) + 1 + ry + O_2(x,y).
\end{aligned}
\end{equation}
In order to get an actual Fatou coordinate, we
consider the functions
\begin{equation}\label{eq_def_fai_0n}
\t \fai_{0,n}:= \t w^{\iota} (F_0^{n}(x,y)) - n.
\end{equation}
The following Lemma proves that the $\t \fai_{0,n}$'s
converge to an actual Fatou coordinate $\t \fai$
as $n\to \infty$.
\begin{lemma}\label{lemma_prima_coord}
The functions $\t \fai_{0,n}$ converge,
locally uniformly on $\t C_0$,
to an analytic function $\t \fai: \t C_0 \to \C$ satisfying
\[
\t \fai (F_0 (p)) = \t \fai (p) + 1.
\]
\end{lemma}

\begin{proof}
Set $A_0 (x,y) :=
\t w_0^\iota (F_0(x,y))-
\t w_0^\iota (x,y)
-1
=
\t \fai_{0,1} (x,y)- \t \fai_{0,0} (x,y)$ and
notice that
$A_0 (F^n_0 (x,y))=
\t \fai_{0,n+1} (x,y)- \t \fai_{0,n} (x,y) $.
In order to ensure the convergence of the $\t \fai_{0,n}$'s
we can prove that the series of the $A_0 (F^n_0 (x,y))$'s
converges normally on $\t C_0$.
It follows from \eqref{eq_conto_w0} that 
\[
A_0 (F^n_0 (x,y))
 =  r y (F^n_0 (x,y)) + O_2(x(F^n_0 (x,y)),y (F^n_0 (x,y))).
\]
By Proposition \ref{prop_hakim_26}, we have $\abs{x(F^n_0 (x,y))}\leq 2/n$ and $\abs{y(F^n_0 (x,y))}\leq 1/n^{\alpha+1}$, for
some $\alpha >0$. This implies that the series $\sum_{n=0}^\infty \abs{ A_0 (F^n (x,y))}$ converges
normally
to
\[
\t \fai (x,y):= \t \fai_0 (x,y) + \sum_{n=0}^{\infty} A_0 (F^n_0 (x,y)).
\]
The functional relation is also easily verified, since $\abs{A_0 (F^n(x,y))}\to 0$.
\end{proof}

In the repelling basin the situation is completely analogous.
  Setting $\t w^o_0 := -\frac{1}{x} - q \log(x)$
on $-\t C_0$
and $\t \fao_{0,n} := \t w^o_{0} (F^{-n}_0 (x,y)) + n$, we have
$\t \fao_{0,n}\to \t \fao$ locally uniformly on $-\t C_0$, where $\t \fao: \t C_0\to \C$ satisfies
the functional relation
$\t \fao \circ F_0 (p)= \t \fao (p)+1$.

We notice that the Fatou coordinates are not unique. For instance,
we can add any constant to them and still have a coordinate satisfying the desired functional relation.
In the following (and in Theorem \ref{teo_lav_2d}), we shall use as coordinate the one
obtained in Lemma \ref{lemma_prima_coord}
above.

\section{The perturbed Fatou coordinates}\label{section_perturbed}

We consider now the perturbation \eqref{eq_family} of the system
$F_0$. The goal of this section is
modify the Fatou coordinate $\t \fai$ built
in Section \ref{section_prelim} to an approximate coordinate
for $F_\eps$. More precisely, we are going to construct
some
coordinates $\t \fai_\eps$
(with values in $\C$)
that, on suitable subsets of $\t C_0$:
\begin{enumerate}
 \item almost conjugate $F_\eps$ to a translation by
$1$,
in the sense that the error that
we have in considering $F_\eps$
 as a translation in this new chart will be bounded and explicitly estimated; and
\item tend to the one-dimensional Fatou coordinates $\t \fai$ for $F_0$ as $\eps \to 0$.
\end{enumerate}

We shall only be concerned with
$\eps$ small and
satisfying
\begin{equation}\label{eq_condizioni_eps}
\begin{cases}
\Re \eps >0\\
\abs{\Im \eps} < c\abs{\eps^2}.
\end{cases}
\end{equation}
Notice that this
means that $\eps$
is contained in the region, in a neighbouhood of the origin, of the points with positive real part and
bounded by two circles
of the same radius
 centered on the
 imaginary axis
and tangent one to the other at the origin.
Notice in particular that, by definition, every sequence $\eps_\nu$ associated to an $\alpha$ sequence
$\pa{\eps_\nu,n_\nu}$ (see Definition \ref{defi_alpha_sequence})
satisfies the above property.

First of all, we
fix
a small neighbourhood $U$
of the origin, such that $F_\eps$ is invertible in $U$, for $\eps$
sufficiently small.
In this section, we shall only be concerned with this local situation.
Then, fix sufficiently small $\gamma<\gamma', R$ and $s$
such that Proposition \ref{prop_hakim_26} holds on $\t C_0 (\gamma, R, s)$ and 
$\t C_0 (\gamma', R, s)$
for both $F_0$ and $H_0 := - F^{-1}_0$.
By taking $\gamma$
and $\gamma'$
sufficiently close, we can assume that
$F_0 (\t C_0 (\gamma',R,s) )$ and $H_0 (\t C_0 (\gamma',R,s) )$
are contained in $\t C_0$.
Denote by $\t C_0,\t C'_0  \subset U$
(dropping for simplicity the dependence on the parameters)
these
sets
and by $C_0,C'_0$ their projections on the $x$-plane. We shall assume that $R \rho \ll 1$, and so
that
$\t C_0 \subset \t C'_0 \Subset U$.

We consider the classical 1-variable change of coordinates on $x$ (and depending on $\eps$)
given by
\begin{equation}\label{eq_def_u_eps}
u_\eps (x) = \frac{1}{\eps} \arctan \pa{\frac{x}{\eps}}= \frac{1}{2i\eps} \log \pa{ \frac{i\eps-x}{i\eps +x} }.
\end{equation}

The geometric idea behind this map is the following: for $\eps$
small as in \eqref{eq_condizioni_eps},
$u_\eps$ sends
$i\eps$ to the ``infinity above" and $-i\eps$ to the ``infinity below". Circular arcs connecting
these two points are sent to parallel (and almost vertical) lines. In particular, the image
of the map $u_\eps$
is contained in the strip $\left\{-\frac{\pi}{2\abs{\eps}} < \Re \pa{\frac{\eps}{\abs{\eps}}w} < \frac{\pi}{2\abs{\eps}}\right\}$
and the image of the disc of radius
$\eps$ centered at the origin is the
strip $\left\{-\frac{\pi}{4\abs{\eps}} < \Re \pa{\frac{\eps}{\abs{\eps}}w} < \frac{\pi}{4\abs{\eps}}\right\}$.
Notice that the inverse of this function on
$\left\{-\frac{\pi}{2\abs{\eps}} < \Re \pa{\frac{\eps}{\abs{\eps}}w} < \frac{\pi}{2\abs{\eps}}\right\}$
is given by $w \mapsto \eps \tan \pa{\eps w}$.
We gather in the next Lemma the main properties of $u_\eps$
that we shall need in the sequel.

\begin{lemma}\label{lemma_pr_u_eps} Let $u_\eps$
be
given by
\eqref{eq_def_u_eps}. Then the following hold.
\begin{enumerate}
 \item\label{lemma_pr_u_eps_1} For every compact subset $\Cc \subset  C_0$
 there exist two positive constants $M^- (\Cc)$ and $M^+ (\Cc)$
 such that, for every $x\in \Cc$, we have
 \begin{equation}\label{eq_MM}
-\frac{\pi}{2\abs{\eps}} + M^- < \Re \pa{
 \frac{\eps}{\abs{\eps}}u_\eps (x)
}
<
- \frac{\pi}{2\abs{\eps}} + M^+
\end{equation}
for every $\eps$
sufficiently small.
 \item\label{item3_lemma_pr_u_eps} If $-\frac{\pi}{2\abs{\eps}}< \Re \pa{ \frac{\eps}{\abs{\eps}} u_\eps (x)} < -\frac{\pi}{4\abs{\eps}}$,
 then
 $\abs{x} \leq \frac{1}{\frac{\pi}{2\abs{\eps}} + \Re (\frac{\eps}{\abs{\eps}}u_\eps (x))}$. 
 \end{enumerate}
\end{lemma}
\begin{proof}
 For the first assertion the main point is to notice that, by the compactness of $\Cc$, we have
 \[
u_\eps (x) + \frac{\pi}{2\eps} \to -\frac{1}{ x} 
 \]
 uniformly on $\Cc$, as $\eps \to 0$.
From this we deduce the existence of constants $M^-, M^+$
such that \eqref{eq_MM} holds for every $x\in \Cc$.

For the
second one,
we exploit
the
inverse of $u_\eps$ on $\{-\frac{\pi}{2\abs{\eps}} < \Re \pa{\frac{\eps}{\abs{\eps}}w} < \frac{\pi}{2\abs{\eps}}\}$, which
is given by $w\mapsto \eps \tan (\eps w)$.
We have
\[
\frac{\pi}{4}< \abs{\Re w}<\frac{\pi}{2}
\Rightarrow \abs{\tan w} \leq \tan \abs{\Re w} <\frac{1}{\frac{\pi}{2}- \abs{\Re w}}
\]
and
the
assertion follows putting $w=\eps u_\eps (x)$.
\end{proof}

We define now, by means of the functions $u_\eps$, different regions
in the dynamical plane. In order to do this, we have to define some constants (independent on $\eps$)
that we shall repeatedly
use in the sequel.

First of all, fix some
$1< \rho' < \rho$.
Then, fix some $1<\rho''<5/4$ such that
\[
\abs{\frac{ 4\pi (\rho''-1)}{\tan \pa{ 4
\pi (\rho''-1)}}}>\frac{1}{\rho'}.
\]
This is possible since $\rho' >1$. In particular, $\rho''$ may be very close to 1.
Finally, set
\begin{equation}\label{eq_def_Ktau}
K:= 2\pi (\rho''-1) \mbox{ and } \tau:= \abs{\tan\pa{-\frac{\pi}{2} + \frac{K}{2}}}.
\end{equation}
Without loss of generality, we can take $\rho''$ small enough to ensure
that $K\leq \pi/4$.
Moreover, we shall assume that $\gamma'$ and $s$ are small enough such that
\begin{equation}\label{eq_assumpt_gamma_s}
\begin{cases}
\rho'< \rho \frac{1-\gamma'}{\sqrt{1+\gamma'^2}},\\
4 \tau s <1.
\end{cases}
\end{equation}

Denote by $D_\eps$ the subset of $\C$
given by
\begin{equation}\label{eq_zona_de}
x\in D_\eps \Leftrightarrow
-\frac{\pi}{2\abs{\eps}} + \frac{K}{\abs{\eps}}
< \Re \pa{ \frac{\eps}{\abs{\eps}}u_\eps (x)}<
\frac{\pi}{2\abs{\eps}} - \frac{K}{2\abs{\eps}}.
\end{equation}
Notice the asymmetry in the definition of $D_\eps$. This will be explained
in Lemma \ref{lemma_stima_verticale_globale}.

Let us now move to $\C^2$.
Let $\t D_\eps$
be the product $D_\eps \times  \D_{ 2 e^{4\pi \rho \tau} \abs{\eps}} \subset \C^2$
(the constant $e^{4\pi \rho \tau}$ will be explained in Proposition \ref{prop_stima_y_centro}).
By definition, since $K\leq\pi/4$, we have
\begin{equation}\label{eq_inclusioni_Deps}
\D_{\aeps} \times  \D_{2 e^{4\pi \rho \tau} \abs{\epsilon}}
\subset
\t D_\eps
\subset
 \D_{\tau \aeps} \times \D_{2 e^{4\pi \rho \tau}  \aeps}.
 \end{equation}
Notice in particular
that the ratios $\tau$ and $2 e^{4\pi \rho \tau}$ are
independent of $\eps$.

Set $C_\eps := \frac{\eps}{\aeps}C_0 \setminus D_\eps$ and
$\t C_\eps := \pa{\frac{\eps}{\aeps}, 1}\cdot \t C_0 \setminus \t D_\eps$
the rotations of $C_0$
and $\t C_0$
of $\frac{\eps}{\aeps}$ around the $y$ plane.
Notice that $\t C_\eps \to \t C_0$ and $\t C_\eps \cup \t D_\eps \to \t C_0$
as $\eps \to 0$. Morevover, we have
$\t C_\eps \subset \t C'_0$
for $\eps$ sufficiently small (and satisfying \eqref{eq_condizioni_eps})
The following Lemma
will be very useful in the sequel.

 \begin{lemma}\label{lemma_invarianza_inizio}
 For $\eps$ sufficiently small,
we have
$F_\eps (\t C_\eps) \subset \t C_\eps \cup \t D_\eps$.
\end{lemma} 

\begin{proof}
By the choice of $\t C_0$ and $\t C'_0$, we have $F_0 (\t C'_0) \subset \t C_0$.
Moreover, $F_\eps = F_0 + O(\eps^2) $  and $F_\eps$ uniformly converges
to $F_0$ on compact subsets of $\t C'_0$.
The assertion then follows from the 
the first inclusion in \eqref{eq_inclusioni_Deps}.
\end{proof}

The first step in the construction of the almost Fatou coordinates
consists in
considering
the functions $\t u_\eps$ given by
\[\t u_\eps (x,y) := u_\eps (x).\]
The following lemma gives the fundamental estimate on $\t u_\eps$: in this chart, the map $F_\eps (x,y)$
approximately acts as a translation by 1 on the first coordinate. Here and in the following, it will be useful to consider
the expression
\[
\gamma_\eps (x,y) := \frac{\alpha_\eps (x,y)}{1+ x \alpha_\eps (x,y)}.
\]
It is immediate to see that $\gamma_\eps (x,y) = 1 + qx + ry + O_2 (x,y) + O(\eps^2)$.

\begin{lemma}\label{lemma_trasl}
Take
 $p = (x,y) \in \t C_\eps \cup \t D_\eps$. Then
 \[
 \t u_\eps (F_\eps (p)) - \t u_\eps (p) = 1 + qx + ry + O_2(x,y) + O(\eps^2).
 \]
 In particular, when
$\gamma, R, s $ and $\eps(\gamma, R, s)$
are small enough,
 for $p = (x,y) \in \t C_\eps \cup \t D_\eps$
 we have
 \[
\abs{  \t u_\eps (F_\eps (p)) - \t u_\eps (p) -1}< \rho'' -1
\mbox{ and }
\abs{  \frac{\eps}{\aeps}  \pa{ \t u_\eps (F_\eps (p)) - \t u_\eps (p) }-1}< \rho'' -1.
 \]
\end{lemma}

\begin{proof}
Since
$x(F_\eps (x,y)) = x + (x^2 + \eps^2)\alpha_\eps (x,y)$, it follows that
\[
\frac{i\eps - x (F_\eps(x,y))}{i\eps + x(F_\eps (x,y))} =
\frac{
(i\eps - x)\pa{
1+ (x+i\eps) \alpha_\eps (x,y)
}
}{
(i\eps + x)\pa{
1+ (x-i\eps) \alpha_\eps (x,y)
}
}
\]
and so
\[
\frac{i\eps + x}{i\eps -x}
\frac{i\eps - x (F_\eps(x,y))}{i\eps + x(F_\eps (x,y))} 
=
\frac{1+i \eps \gamma_\eps (x,y)}{1-i\eps \gamma_\eps (x,y)}.
\]
The desired difference is then equal to
\[
\begin{aligned}
\t u_\eps (F_\eps (p)) - \t u_\eps (p) & =
\frac{1}{2i\eps} \log \frac{1+i \eps \gamma_\eps (x,y)}{1-i\eps \gamma_\eps (x,y)}\\
& =
\frac{1}{i\eps} \bra{ i\eps \gamma_\eps (x,y) + \frac{1}{3} \pa{i\eps \gamma_\eps (x,y)}^3 + O (\eps^4)}\\
& =
\gamma_\eps (x,y) + O(\eps^2)\\
&=
1+ qx + ry + O_2 (x,y) + O(\eps^2)
\end{aligned}
\]
and the assertion is proved.
\end{proof}

The next step is
to slightly modify our coordinate $\t u_\eps$ to a coordinate $\t w_\eps^\iota$ satisfying the
following two properties:
\begin{enumerate}
\item $\t w^\iota_\eps
\to \t w_0^\iota$ (with $\t w_0^\iota$ as in \eqref{eq_def_w0}) as $\eps\to 0$, and
\item $\t w^\iota_\eps (F^n_\eps (p)) - n \to \tilde \fai$ when $\eps\to 0$ and $n\to \infty$
satisfying some relation to be determined later.
\end{enumerate}
We also look for functions $\t w_\eps^o$ satisfying analogous properties on $-\t C_0$.
Recall that
the functions $\t w_0^\iota (x,y)$
and 
$\t w_0^o (x,y)$
almost semiconjugates the (first coordinate of the) system $F_0$
to a translation by 1
(by \eqref{eq_conto_w0}).

We set
\[
\tilde w_\eps (x,y) := \t u_\eps (x,y) - \frac{q}{2}\log(\eps^2 + x^2)
= \frac{1}{2i\eps} \log \pa{ \frac{i\eps-x}{i\eps +x} } - \frac{q}{2}\log(\eps^2 + x^2).
\]
and consider
their incoming
and outgoing
normalizations $\t w^\iota_\eps$ and $\t w_\eps^o$
given by
\[
\begin{aligned}
\tilde w^{\iota}_\eps (x,y) & := \frac{1}{2i\eps} \log \pa{ \frac{i\eps-x}{i\eps +x} } - \frac{q}{2}\log(\eps^2 + x^2)
+ \frac{\pi}{2\eps},\\
\tilde w^{o}_\eps (x,y) & := \frac{1}{2i\eps} \log \pa{ \frac{i\eps-x}{i\eps +x} } - \frac{q}{2}\log(\eps^2 + x^2)
- \frac{\pi}{2\eps}.
\end{aligned}
\]
It is immediate to check that the first request is satisfied, i.e., that
$\tilde w_\eps^\iota (x,y)\to \t w_0^\iota $ on $\t C_0$
(and $\tilde w_\eps^o (x,y)\to \t w_0^o $ on $-\t C_0$)
as $\eps\to 0$.
In the next proposition we
estimate the distance between the reading of $F_\eps$ in this new chart $\t w_\eps$ and the translation by 1.
We want to prove, in particular,
that now the error has no linear terms in the $x$ variable.
Indeed, notice that also for the system $F_0$
we had to remove this term (see Lemma \ref{lemma_prima_coord})
to ensure the convergence
of the series of the $A_0 (F^n_0 (p))$'s,
by the harmonic behaviour of $x(F_0^n (p))$.
For convenience of notation, we denote this error
by
\[
A_\eps (x,y):= \tilde w_\eps (F_\eps (x,y) ) - \tilde{w} (x,y) -1
\]
We then have the following estimate.
\begin{prop}\label{prop_error_A}
 $A_{\eps} (x,y) = ry + O_2 (x,y) + O(\eps^2)$.
\end{prop}

Notice that, differently from \cite{bsu}, here the error is still linear in $y$. The reason is that we do not add any correction
term in $y$ in the expression of $\tilde w_\eps$.
On the other hand, 
by our assumptions we do not have any linear dipendence in $\eps$.

\begin{proof}
 The computation is analogous to the one in \cite{bsu}.
 By the definition of $\t w_\eps$ and the analogous property of $\t u_\eps$ (Lemma \ref{lemma_trasl})
 we have
 \[
 \begin{aligned}
 \tilde w_\eps (
 F_\eps (x,y) ) - \tilde{w_\eps} (x,y)  & = &&
  \tilde u_\eps (
  F_\eps (x,y) ) - \tilde{u_\eps} (x,y)\\
& && - \frac{q}{2}\log(\eps^2 + x(
F_\eps (x,y))^2)
  + \frac{q}{2}\log(\eps^2 + x^2)\\
  & = &&
    1 + qx + ry + O_2 (x,y) + O(\eps^2)\\ & && -\frac{q}{2} \log \frac{\eps^2 + x(
    F_\eps (x,y))^2}{\eps^2 + x^2}.
 \end{aligned}\]
 It is thus sufficient to prove that
 \[
 \frac{\eps^2 + x(
 F_\eps (x,y))^2}{\eps^2 + x^2} = 1+2x + O_2 (x,y) + O (\eps^2).
 \]
 But
\[
\begin{aligned}
\eps^2 + x(
F_\eps (x,y))^2
&=
\eps^2 + x^2 + (x^2 + \eps^2)^2 \alpha_\eps^2 (x,y) +2x(x^2 + \eps^2)\alpha_\eps (x,y)\\
&= (x^2 + \eps^2 )(1+ 2x \alpha_\eps (x,y) + O(x^2,\eps^2))\\
&=
(x^2 +\eps^2) (1+ 2x + O_2 (x,y)+ O(\eps^2))
\end{aligned}
\] 
and the assertion follows. 
\end{proof}

Let us finally introduce the \emph{incoming almost Fatou coordinate},
by means of the $\t w^\iota_\eps$, as it was done for the map $F_0$ in \eqref{eq_def_fai_0n}.
Set
\begin{equation}\label{eq_defi_almost_coord}
\t \fai_{\eps, n} (p):=
\tilde w^\iota_\eps (F^n_\eps (p) ) -n
= \tilde w^\iota_\eps (p ) + \sum_{j=0}^{n-1}A_\eps (F^j_\eps (p)).
\end{equation}

We shall be particularly interested in the following relation between the parameter
$\eps$ and the number of iterations.

\begin{defi}\label{defi_bounded}
A sequence
$(\eps_\nu, m_\nu) \subset (\C \times \N)^\N$ such that $\eps_\nu \to 0$
will be said \emph{of bounded type} if
$\frac{\pi}{2\eps_\nu}-m_\nu$ is bounded in $\nu$. 
\end{defi}
Notice
that, given an $\alpha$-sequence $(\eps_\nu, n_\nu)$, the sequence $(\eps_\nu, n_\nu/2)$
is of bounded type.

The following result in particular proves that the coordinates
$\t w^\iota_\eps$ satisfy
the second request. This convergence will be crucial in order to prove
Theorem \ref{teo_lav_2d}.
Here $\t \fai$
denotes the Fatou coordinate on $\t C_0$
given by Lemma \ref{lemma_prima_coord}.

\begin{teo}\label{lemma_conv_fai}
 Let $(\eps_\nu, m_\nu)_{\nu \in \N}$ be a sequence of bounded type.
 Then
 \[
 \t \fai_{\eps_\nu, m_\nu} \to \t {\fai}
 \]
locally uniformly on $\t C_0$.
 \end{teo}

We can also define the outgoing almost Fatou coordinates on $-\t C_0$
as
 \[
\t \fao_{\eps, n} (p):= \t w^o (F^{-n}_\eps (p)) + n
 \]
(recall that by assumption $-\t C_0$ is contained in a neighbourhood
$U$ of the origin where $F_\eps$ is invertible, for $\eps$ sufficiently small).
 The following convergence is then an immediate consequence of Theorem \ref{lemma_conv_fai}
 applied to the inverse system.
 
 \begin{cor}\label{cor_conv_fao}
  Let $(\eps_\nu, m_\nu)$ be a sequence
 of bounded type. 
 Then
  \[
 \t  \fao_{\eps_\nu,m_\nu}\to \t {\fao}
  \]
  locally uniformly on $-\t C_0$.
 \end{cor}

To prove Theorem \ref{lemma_conv_fai}, we need to estimate
the series of the errors in \eqref{eq_defi_almost_coord}. In particular,
we need to bound the modulus of the two coordinates of the orbit $F_\eps^j (p)$, for $p\in \t \C_0$
and $j$ up to (approximately) $\pi/2\aeps$. This is the content of the next section. The proof
of Theorem \ref{lemma_conv_fai} will be then given in Section \ref{section_prelim_convergence}.

In our study, we will need to carefully
 compare the behaviour of $F_\eps$ in $\t C_0$ and the one of $F^{-1}_{\eps}$
 on $-\t C_0$.
 Notice that $F_\eps^{-1}$ is given by
\[
F^{-1}_\eps 
\left(
\begin{array}{c}
 x\\
 y
\end{array}\right)
=
\left(\begin{array}{c}
       x - (x^2 + \eps^2) (1+   (q-1)x + r y + + O(\eps^2) + O_2(x,y))\\
       y(1 - \rho x + O(\eps^2) + O_2(x,y))
      \end{array}\right)
\]
In order to compare the behaviour of the orbits for $F^{-1}_\eps$
with the ones for $F_\eps$, it will be useful
to consider the change of coordinate $(x,y)\mapsto (-x,y)$ and thus
study the maps
 \begin{equation}\label{eq_family_he}
\begin{aligned}
H_\eps 
\left(
\begin{array}{c}
 x\\
 y
\end{array}\right)
& =
\left(\begin{array}{c}
       x + (x^2 + \eps^2) (1+   (-q+1)x + r y + + O(\eps^2) + O_2(x,y))\\
       y(1 + \rho x + O(\eps^2) + O_2(x,y))
      \end{array}\right)\\
& =      \left(
\begin{array}{c}
 x+(x^2 +\eps^2) \alpha^H_\eps (x,y)\\
       y (1+ \rho x + \beta^H_\eps (x,y))
\end{array}
\right)
 \end{aligned}      
\end{equation}
In this way, we can study both $F_\eps$ and $H_\eps$
in the same region of space. Notice that
the main difference between $F_\eps$ and $H_\eps$ is that the coefficient $q$ has changed sign.

\section{The estimates for the points in the orbit}\label{section_fasi}

In this section we are going to study the 
orbit of a point $p\in \t C_0$
under the iteration of $F_\eps$. In particular,
since the main application we have in mind
is the study of $F^{n_\nu}_{\eps_\nu}$
when $(\eps_\nu, n_\nu)$
is an $\alpha$-sequence, we shall be primarily interested
in the study of orbit up to an order
of
$\pi/\aeps$ iterations.

Recall that the set $\t C_0$ is given by Proposition
\ref{prop_hakim_26} and in particular consists of points that converge
to the origin under $F_0$ tangentially to the (negative) real axis of the complex direction $[1:0]$.
We shall still assume (by taking $R \ll 1$ small enough)
that $\t C_0$ is contained in a small neighbourhood $U$ of the origin
where $F_0$ and $F_\eps$
are invertible, for $\eps$
sufficiently small.

By Lemma \ref{lemma_pr_u_eps},
for every compact $\Cc\subset \t C_0$
there exist two constants $M^- (\Cc)$
and $M^+ (\Cc)$ such that
\begin{equation}\label{eq_MM_section_coord}
 -\frac{\pi}{2\aeps} + M^-(\Cc) \leq \Re \pa{\frac{\eps}{\aeps} \t u_\eps(p)}\leq  -\frac{\pi}{2\aeps} + M^+(\Cc)
 \quad
 \forall p \in \Cc, \forall \eps \leq \eps_0.
\end{equation}
Without loss of generality, we will assume that $M^-$
and $M^+$
are integers and $\gg 1$ (since $R\ll 1$).

We shall divide the estimates of the coordinates
of $F_\eps^j (p)$
according to its position with respect to the set
$\t D_\eps$, i.e., according to the position of $x(F^j_\eps (p))$ with respect to $D_\eps$
as in \eqref{eq_zona_de}.
The following notation will be consistently used through all our study.

\begin{defi}\label{defi_entering}
Given $p\in \t C_0$ and $\eps$ such that $p\in \t C_\eps$,
we define the \emph{entry time} $n_p (\eps)$ and the \emph{exit time} $n'_p (\eps)$ by
\begin{equation}\label{eq_def_entering}
\begin{aligned}
& n_p (\eps) := \min \set{ j \in \N \colon F_\eps^j (p) \in \tilde D_\eps }\\
& n'_p (\eps): = \min\set{j\in \N\colon F^j_\eps \notin \t C_\eps \cup \t D_\eps}
\end{aligned}
\end{equation}
\end{defi}

The next Proposition gives the bounds
on $n_p (\eps)$ that we shall need in the sequel.

\begin{prop}\label{prop_prima_stime_1}
 Let
$\Cc \subset \t C_0$
be a compact subset and $M^-, M^+$
be as in \eqref{eq_MM_section_coord}.
Then,
 for every $p=(x,y)  \in \Cc$
 and $\eps$ sufficiently small,
\[
 \frac{K}{\rho'' \aeps} - \frac{M^+}{\rho''}
 \leq n_p (\eps) \leq \frac{K}{(2- \rho'')\aeps} - \frac{M^-}{2-\rho''}.
\]
In particular,
$F^j_\eps (p) \in  \t C_\eps$ for $0 \leq j < \frac{K}{\rho''\aeps} - \frac{M^+}{\rho''}$.
\end{prop}

\begin{proof}
Notice that, since $F_\eps (\tilde C_\eps) \subset \tilde C_\eps \cup \tilde D_\eps$ (by Lemma  \ref{lemma_invarianza_inizio}),
we only have to study the
first coordinate of the orbit.
Since $\t C_\eps \to \t C_0$, we have that
$\Cc \subset \t C_\eps$
 for $\eps$ sufficiently small.
From Lemma \ref{lemma_trasl} it follows
that
\[
2-\rho'' <
\Re \pa{
\frac{\eps}{\aeps}
\t  u_\eps (F_\eps (p))
}
-
\Re
\pa{
\frac{\eps}{\aeps} 
\t u_\eps (p)
}
< \rho''.
\]
Thus, we
deduce that
\begin{equation}\label{eq_stima_reu}
-\frac{\pi}{2\aeps} + M^-
+ (2-\rho'')j
<
\Re \pa{
\frac{\eps}{\aeps}
\t u_\eps (F^j_\eps (q))
}
<
-\frac{\pi}{2\aeps}
+ M^+ +
\rho'' j
\end{equation}
and the assertion
follows from the definition of $D_\eps$ (see \eqref{eq_zona_de}).
\end{proof}

\subsection{Up to $n_p (\eps)$}
Given $p$ in some compact subset $\Cc \in \t C_0$, here we study the modulus of
the two coordinates of the points in the orbit for $F_\eps$ of
$p$ until they fall in $\t D_\eps$, i.e., for a number of iteration up to $n_p (\eps)$.
We start estimating
the first coordinate. Here we shall make use of the definition of $K$ (see \eqref{eq_def_Ktau}).

\begin{lemma}\label{lemma_stima_x_modulo}
Let
$\Cc \subset \t C_0$
be a compact subset and $M^-$
be as in \eqref{eq_MM_section_coord}. Then
 \[
 \abs{x(F^j_\eps (p))} \leq \frac{2}{j+ M^-}
 \]
 for every $p\in \Cc$,
 for $\eps$ small enough and $j\leq n_p (\eps)$.
\end{lemma}

\begin{proof}
The statement
follows from Lemma \ref{lemma_pr_u_eps} \eqref{item3_lemma_pr_u_eps}
and the (first) inequality in \eqref{eq_stima_reu}.
Indeed, we have (recall that $3/4 < 2-\rho''<1$)
\[
\begin{aligned}
\abs{x(F^j_\eps(p))}
&
<\frac{1}{
\frac{\pi}{2\aeps}+\Re ( \frac{\eps}{\aeps} \t u_\eps (F^j_\eps (p)))
}\\
&\leq
\frac{1}{
\frac{\pi}{2\aeps} - \frac{\pi}{2\aeps}
+ (2-\rho'')j + M^-
}
\leq \frac{1}{2-\rho''} \frac{1}{j+ M^-}
\leq \frac{2}{j+ M^-}.
\end{aligned}
\]
and the inequality is proved.
\end{proof}

We now come to the second coordinate. Estimating
this is the main difference between our
setting and the semiparabolic one.
Notice that, by \eqref{eq_family},
in order to bound the terms $\abs{y (F^j_\eps (p))}$, we will need to get an estimate
from below of the first coordinate.
This will be done by means of the following lemma.

\begin{lemma}\label{lemma_confronto_modello}
Let
$\Cc \subset \t C_0$
be a compact subset and $M^-$
be as in \eqref{eq_MM_section_coord}.
Let $p,q \in \Cc$ and
set $q_j:= \eps\pa{\tan \eps (\t u_\eps (q) + j) }$
and $\t q_j:= \eps\pa{\tan \eps (\t u_\eps (q) + \aeps j/\eps) }$.
Then,
for some positive
constants $C$ depending on $\Cc$ and $C_\eps$
depending on $\Cc$
and $\eps$, and going to zero as $\Re \eps\to 0$,
\begin{equation}\label{eq_stima_bootstrap}
\abs{ x(F^j_\eps (p))-  q_j}< C \frac{ 1+\log (M^- + j)}{(M^- +j)^2}
\end{equation}
and
\begin{equation}\label{eq_stima_bootstrap_2}
\abs{ x(F^j_\eps (p))-  \t q_j}< C \frac{ 1+\log (M^- + j)}{(M^- +j)^2} + C_\eps \frac{1}{M^+ + j}
\end{equation}
for every $0\leq j \leq
n_p (\eps)$.
\end{lemma}

Notice in particular that the two estimates reduce to the same for $\eps$ real.

\begin{proof}
The idea is to first estimate the distance between the two sequences
$\t  u_\eps (F^j_\eps (p))$ and $\t u_\eps (q) + j$ (and between $\t  u_\eps (F^j_\eps (p))$ and $\t u_\eps (q) + \aeps j/\eps$)
and then to see how this distance is transformed by the application 
of the inverse of $u_\eps$.
Notice that, since $j\leq n_p(\eps)$,
by definition of $\t D_\eps$ (see \eqref{eq_zona_de})
we have $\Re \pa{ \frac{\eps}{\aeps} \t u_\eps (F^j_\eps (p)) }<- \frac{\pi}{4\aeps}$ for the points in the orbit under
consideration (since $K\leq \pi/4$).

We
first
prove that
\begin{equation}\label{eq_stima_carta_u_num}
\abs{\t  u_\eps (F^j_\eps (p)) -  \t u_\eps (q) - j} \leq C_1 \pa{1+\log (M^- +j)}.
\end{equation}
Notice that this is an improvement with respect to the
estimate obtained in Lemma \ref{lemma_trasl},
but that we shall need both that estimate and the bound from above obtained in Lemma \ref{lemma_stima_x_modulo}
in order to get this one.

By the definition of $M^-$,
we have 
that $\abs{x(p)}$ and $\abs{x(q)}$
are bounded above by $2/M^-$.
Recalling that $\abs{y}\leq s\abs{x}$ for every $(x,y)\in \t C_\eps$,
Lemma \ref{lemma_trasl} gives
\[
\abs{\t u_\eps (F^j_\eps (p)) - \t u_\eps (p)-j} \leq
 c_1 \sum_{i<j} \abs{x (F^i_\eps (p))} + c_2 \sum_{i<j} \pa{  \abs{x (F^i_\eps (p))}^2 + \abs{\eps}^2}.
\] 
Since by Lemma \ref{lemma_stima_x_modulo}
we have $\abs{x (F^j_\eps (p))}\leq2/(j+M^-)$
and the maximal number of iterations $n_p (\eps)$
is
bounded by a constant times $1/\aeps$,
 this gives
\[
\abs{\t u_\eps (F^j_\eps (p)) - \t u_\eps (p) - j} \leq C_2 \pa{1+\log (M^- +j)}
\] 
for some positive $C_2$, 
and the estimate \eqref{eq_stima_carta_u_num} follows since the two sequences
$\pa{\t u_\eps (p) + j}_j$ and $\pa{\t u_\eps (q) + j}_j$
obviously stay at constant distance.

We then consider the sequence $\t q_j$. Using \eqref{eq_stima_carta_u_num}, it is immediate to see that
\begin{equation}\label{eq_stima_carta_u_num_2}
\abs{\t  u_\eps (F^j_\eps (p)) -  \t u_\eps ( q) - \aeps j/\eps} \leq C_1 \pa{1+\log (M^- +j)} + \abs{\arg(\eps)} j,
\end{equation}
since 
the distance between
the two sequences
$\t u_\eps ( q) +  j$ and $\t u_\eps ( q) +\aeps j/\eps$.
is bounded by the last term.

We now need to estimate how the errors in \eqref{eq_stima_carta_u_num} and \eqref{eq_stima_carta_u_num_2}
are transformed when passing to the dynamical space,
and
in particular recover the quadratic denominator in \eqref{eq_stima_bootstrap}.
By \eqref{eq_stima_carta_u_num_2}
we have
\[
\begin{aligned}
 \Re \pa{ \frac{\eps}{\aeps} \t u_\eps (F^j_\eps (p)) }
 & \geq -\frac{\pi}{2\aeps} +M^- + j - C_1 \pa{1+\log (M^- +j)} - \abs{\arg \eps} j\\
 & > - \frac{\pi}{2\aeps} + C_3 (M^- + j)
\end{aligned}
 \]
for $\eps$
sufficiently small (as in \eqref{eq_condizioni_eps}), $j\leq n_p (\eps)$ and
some
$C_3>0$.
So, given
$L >0$, it is enough to
bound from above the modulus of the derivative
of the inverse of $u_\eps$ on the
strip
$\left\{-\frac{\pi}{2\aeps} + L \leq \Re \pa{ \frac{\eps}{\aeps}w}< -\frac{\pi}{4\aeps}\right\}$
by (a constant times) $1/\abs{L }^2$.
This can be done with a straightforward computation.
Recall that $u_\eps (z) = \frac{1}{\eps}\arctan \pa{\frac{z}{\eps}}$, so that 
its inverse is given by $\eps \tan (\eps w)$.
The derivative of this inverse at a point $-\pi/2\eps + w$
is thus given by
$\psi_\eps (w)=\eps^2 \pa{\cos \pa{\eps w}}^{-2}$.
On
the strip in consideration,
$\psi_\eps$
takes its maximum at $w = -\frac{\pi}{2\eps} + L$,
where we have
$\psi_\eps (-\frac{\pi}{2\eps} + L) = \eps^2 / \sin^2 (\eps L)$.
The estimate then follows since
$x\leq 2 \sin(x)$ on $[0,\pi/4]$.
\end{proof}

\begin{prop}\label{prop_stima_prima_x_completa}
Let
$\Cc \subset \t C_0$
be a compact subset,
$M^-, M^+$
be as in \eqref{eq_MM_section_coord} and
$C,C_\eps$
as in Lemma \ref{lemma_confronto_modello}.
Then
 \[
 \pa{\frac{1}{\rho'} - C_\eps }\frac{1}{ M^+
+  j} - C \frac{ 1+\log (M^- + j)}{(M^- +j)^2}\leq \abs{x(F^j_\eps (p))} \leq \frac{2}{j+ M^-}
 \]
 for every $p\in \Cc$,
 for $\eps$ small enough and $j\leq n_p (\eps)$.
\end{prop}

\begin{proof}
The second inequality is the content of Lemma \ref{lemma_stima_x_modulo}.
Let us then prove the lower bound.
By Lemma \ref{lemma_confronto_modello}, it is enough to get the bound
\[
\frac{1}{\rho' (M^+
+  j)} \leq \abs{\t q_j}
\]
where $\t q_j := \eps \tan \pa{\eps ( \Re (\t u_\eps (p))+ \aeps j/\eps)}$
as in Lemma \ref{lemma_confronto_modello}.
Notice that we arranged the points $\frac{\eps}{\aeps}\t u_\eps (\t q_j)$
to be on the real axis.
Since we have $\Re \frac{\eps}{\aeps} \t u_\eps (q_0) < -\frac{\pi}{2\aeps} + M^+$
(and thus 
$
\Re \frac{\eps}{\aeps} u_\eps \pa{\t q_j}
\leq -\frac{\pi}{2\aeps} + M^+  + j$),
it follows that
\[
 \abs{\t q_j}
\geq \aeps  \abs{ \tan \pa{ \eps\pa{-\frac{\pi}{2\aeps} + (M^+ + j) \frac{\aeps}{\eps}}} }
= \frac{\aeps}  {{ \tan \pa{ M^+ \aeps+j\aeps} }}.
\]
We thus have to prove that, for $\eps$ sufficiently small and $j\leq n_p (\eps)$,
\[
\frac{\aeps M^+ +\aeps  j}{{\tan \pa{M^+ \aeps + j \aeps}}} > \frac{1}{\rho'}.
\]
The left hand side is decreasing in $j$, so we can evaluate it at $j= n_p (\eps)$, which is less or equal
than $ \frac{K}{(2-\rho'')\aeps}$
by Proposition \ref{prop_prima_stime_1}.
We thus need to
prove that, for $\eps$ sufficiently small,
\[
\frac{\aeps M^+ + \frac{K}{2-\rho''}}{\abs{\tan \pa{M^+ \aeps + \frac{K}{2-\rho''}}}}
> \frac{1}{\rho'}.
\]
This follows since $\aeps {M^+} + \frac{K}{2-\rho''} <2K$ for $\aeps \ll 1$
and,
by assumption, $K$ satisfies $\abs{\frac{2K}{\tan (2 K)}}> \frac{1}{\rho'}$.
This concludes the proof.
\end{proof}

We can now give the estimate for the second coordinate.

\begin{prop}\label{prop_stima_y_prima}
Let
$\Cc \subset \t C_0$
be a compact subset and $M^+$
be as in \eqref{eq_MM_section_coord}.
 There exists a positive constant
 $c_1$, depending on $\Cc$,
 such that for $p\in \Cc$ and 
$J \leq n_p (\eps)$,
  \[
  \abs{y (F^J_\eps (p))}
  \leq
c_1 
 \abs{y(p)}
  \prod_{l=M^+}^{M^+ + J-1}\pa{ 1 -  \frac{\tilde \rho}{l}}
  \]
for some $1 < \tilde \rho 
< \frac{\rho}{\rho'} \frac{1-\gamma'}{\sqrt{1+\gamma'^2}}$.
\end{prop}

Notice that $1 < \frac{\rho}{\rho'} \frac{1-\gamma'}{\sqrt{1+\gamma'^2}}$
by the assumption \eqref{eq_assumpt_gamma_s}.

\begin{proof}
We shall make use of both estimates obtained in
Proposition \ref{prop_stima_prima_x_completa}.
Since the part of orbit which we are considering is in $\t C_\eps$ (at least) up to $J-1$, 
we have
$\abs{y (F^j_\eps (p))} \leq s \abs{x(F^j_\eps (p))}$ and $\abs{x(F^j_\eps (p))} > \aeps$, for $j\leq J-1$.
So,
by
the expression of $y(F_\eps(p))$
in \eqref{eq_family},
we get
 \[
 \begin{aligned}
 \abs{y(F^{J}_\eps (p))}
 & \leq
  \abs{y(p)}
 \prod_{j=0}^{J-1}\abs{ 1 + \rho x (F^j_\eps (p)) + O(x^2 (F^j_\eps (p))}\\
 & \leq
  \abs{y(p)}
 \prod_{j=0}^{J-1} \pa{\abs{ 1 + \rho x (F^j_\eps (p))} + \t c_1 \abs{x^2 (F^j_\eps (p))}}
 \end{aligned}
 \]
 for some positive $\t c_1$.
 For $\eps$ sufficiently small, we have
$\t C_\eps \subset \t C'_0 = \t C_0 (\gamma', R, s)$
(see Proposition \ref{prop_hakim_26}). This implies 
that
 $\abs{\Im \pa{x (F^j_\eps (p))}}< \gamma' \abs{\Re \pa{x (F^j_\eps (p))}}$
 for every $j< n_p (\eps)$.
Thus
 \[
 \begin{aligned}
 \abs{1 + \rho x (F^j_\eps (p))} & \leq 1-  \rho\abs{\Re \pa{x (F^j_\eps (p))}} + \rho  \abs{\Im \pa{x (F^j_\eps (p))}}\\
 & \leq 1 - \rho (1-\gamma' ) \abs{\Re \pa{x (F^j_\eps (p))}} \\
 & \leq 1- \rho \frac{1-\gamma'}{\sqrt{1+\gamma'^2}}  \abs{x (F^j_\eps (p))}
 \end{aligned}
 \]
 and thus,  
 by the estimates on $x(F^j_\eps (p))$ in Proposition
 \ref{prop_stima_prima_x_completa}
 we deduce that
 (for $\eps$ sufficiently small)
\[
\begin{aligned}
 \abs{y(F^{J}_\eps (p))}
 & \leq 
 \abs{y(p)}
  \prod_{j=0}^{J-1}\pa{ 1 -\rho  \frac{1-\gamma'}{\sqrt{1+\gamma'^2}} \pa{ \frac{1}{\rho'}-C_\eps }\frac{1}{M^+ +j}
  + \t c'_1
  \frac{ 1+\log (M^- + j)}{(M^- +j)^2}
  }\\
 &\leq
 c_1 
 \abs{y(p)}
  \prod_{j=0}^{J-1}\pa{ 1 - \tilde \rho \frac{1}{M^+ +j}}
 \end{aligned}
 \]
where $\t \rho$ is some constant such that $1 < \t \rho < \frac{\rho}{\rho'} \frac{1-\gamma'}{\sqrt{1+\gamma'^2}}$,
 and the assertion follows.
\end{proof}

\subsection{From $n_p (\eps)$ to $n'_p(\eps)$}

Notice that
$\t D_\eps$
needs not to be $F_\eps$-invariant.
In this section  we estimate the second coordinate
for points in an orbit entering $\t D_\eps$ (and in particular
explain the constant $e^{4\pi \rho \tau}$
in the definition of $\t D_\eps$).
Our goal is prove a lower bound on $n'_p (\eps)$ (and moreover to prove that the orbit cannot
come back to $\t C_\eps$). This will in particular give an estimate for the
coordinates of the point in the orbit for $j$ up to the lower bound of $n'_p (\eps)$ (since in $\t D_\eps$
both $\abs{x}$ and $\abs{y}$ are bounded by (a constant times)
$\aeps$).

\begin{prop}\label{prop_stima_y_centro}
 Let $\Cc \subset \t C_0$ be a compact subset.
 Then, for every
 $p\in \Cc$,
  and $n_p (\eps) < j \leq n_p' (\eps)$, we have
  \[\abs{y (F^j_\eps (p))} \leq e^{4\pi\rho\tau} \abs{y (F^{n_p (\eps)}_\eps (p))} \leq e^{4\pi \rho \tau} \aeps\]
\end{prop}

\begin{proof}
Recall that $\tau=\tan\pa{-\frac{\pi}{2} + \frac{K}{2}}$ and that by the assumption \eqref{eq_assumpt_gamma_s}
we have $4 s\tau <1$.
Since
the part of
 orbit under consideration
 is contained in $\t D_\eps$
 (and thus $\abs{x(F_\eps^j (p))} \leq \tau \aeps$, by \eqref{eq_inclusioni_Deps}),
 we have
  \[
  \begin{aligned}
   \abs{y (F^j_\eps (p))} & \leq \abs{y (F_\eps^{n_p (\eps)} (p))}  \prod_{i=n_p (\eps)}^{j-1}
\pa{1 + 2 \rho \tau \aeps}\\
& \leq \abs{y (F_\eps^{n_p (\eps)} (p))}  \prod_{i=n_p(\eps)}^{\lfloor\frac{\pi - K/2}{(2 -\rho'') \aeps}\rfloor}
   \pa{1 + 2 \rho \tau \aeps}.
   \end{aligned}
  \]
 The product is bounded by $(1+2\rho\tau \aeps)^{2\pi/\aeps} \leq e^{4\pi\rho\tau}$ as $\eps\to 0$.
 Moreover, we have
 $\abs{y (F_\eps^{n_p (\eps)} (p))}
 \leq \abs{y (F_\eps^{n_p (\eps)-1} (p))} \abs{1+\rho x(F_\eps^{n_p(\eps)-1})}
 \leq 4 s \tau \aeps < \aeps$.
 This gives the assertion.
\end{proof}

We can now give the estimate on $n'_p (\eps)$.

\begin{prop}\label{prop_stima_lunghezza_centro}
Let
$\Cc \subset \t C_0$
be a compact subset and $M^-, M^+$
be as in \eqref{eq_MM_section_coord}.
Then, for every $p\in \Cc$,
\[
  \frac{\pi - K/2}{\rho'' \aeps} - \frac{M^+}{\rho''} \leq
 n'_p (\eps) \leq \frac{\pi - K/2}{(2- \rho'')\aeps} - \frac{M^-}{2-\rho''}.
 \]
 Moreover, we have
 $\abs{y(F^j_\eps (p)) } \leq e^{4\pi\rho\tau} \aeps$
 for $n_p (\eps) \leq j <  n'_p (\eps)$
 and
 \[\Re \pa{ \frac{\eps}{\aeps}\t u_\eps \pa{F_{\eps}^{n'_p (\eps)}}} \geq \frac{\pi -K}{2\aeps}.\]
 In particular, once entered
 in $\t D_\eps$, the orbit cannot come back to $\t C_\eps$.
\end{prop}

\begin{proof}
By Proposition \ref{prop_stima_y_centro},
the modulus of the
second coordinate of the points of the orbit is bounded by
$e^{4\rho\pi\tau} \aeps$ for $n_p (\eps) < j \leq  n'_p (\eps)$.
Since for $j\leq n_p (\eps)$ it is bounded by $s \abs{x(F^j_\eps (p))}$, the assertion follows from Equation \eqref{eq_stima_reu}.
\end{proof}

\subsection{After $n'_p (\eps)$}

In order to study the behaviour of $F_\eps$ after $\t D_\eps$, we shall make use
of the family $H_\eps$
introduced in \eqref{eq_family_he}.
The following proposition is an immediate consequence of the analogous results for $F_\eps$ (first assertion of
Lemma \ref{lemma_confronto_modello}).
We denote by
$n_p^H (\eps)$
the entry time for $H$ (see Definition \ref{defi_entering}).

\begin{lemma}\label{lemma_confronto_fh}
Let
$\Cc \subset \t C_0$
be a compact subset and $M^-$
be as in \eqref{eq_MM_section_coord}.
Let $p,q$ be contained in some compact subset $\Cc \subset \t C_0$.
Then, for $\eps$ sufficiently small,
\[
\abs{ x(F^j_\eps (p))- x(H^j_\eps (q))}< C \frac{ 1+\log (M^- + j)}{(M^- +j)^2}
\]
for every $0\leq j \leq \min (n_p (\eps), n_q^H (\eps))$, for some positive constant $C$.
\end{lemma}

We will get the estimates on the second coordinate
in this part of the orbit directly in Section \ref{section_convergence}, when
proving Theorem \ref{teo_lav_2d}, by applying Proposition \ref{prop_stima_y_prima} to both
$F_\eps$ and $H_\eps$.

\section{A preliminary convergence:  proof of Theorem \ref{lemma_conv_fai}}\label{section_prelim_convergence}
In this section we prove Theorem \ref{lemma_conv_fai}. Namely,
given 
a sequence $(\eps_\nu,m_\nu)$
of bounded type (see Definition \ref{defi_bounded}),
we prove that
$\t \fai_{\eps_\nu, m_\nu}\to \t \fai$
and $\t \fao_{\eps_\nu, m_\nu}\to \t \fao$,
locally uniformly on
$\t C_0$ and $- \t C_0$,
where $\t \fai$ and $\t \fao$
are the Fatou coordinates
for $F_0$
given by Lemma \ref{lemma_prima_coord}.
Recall that by assumption these two sets are contained in a neighbourhood $U$
of the origin where $F_\eps$ is invertible, for $\eps$ sufficiently small, ans thus in particular
where $\t \fao$ is well defined.
We shall need the following elementary Lemma.

\begin{lemma}\label{lemma_serie}
 Let $a
 \in \R$,
 be
 strictly greater than 1. Then, for every $j_0\geq l_0\geq 1$ such that
 $0<1-\frac{a}{l}
 <1$ for every $l\geq l_0$,
 the series
 \[\sum_{j=j_0}^\infty \prod_{l=l_0}^{j} \pa{1- \frac{a}{l}
 }\]
 converges.
\end{lemma}

Notice that the Lemma is false when $a=1$, since the series reduces to an harmonic one.
In our applications $a$ will essentially be $\rho$, which we assume by hyphotesis
to be
greater than 1.

\begin{proof}
As in \cite[Lemma 4]{weickert98}, let us set $P_j := \prod^j_{l=l_0} (1- \frac{a}{l}
)$
and
notice that the $P_j$'s admit an explicit expression as
\[
P_j =
c\, \frac{\Gamma (j+1 - a)}{\Gamma (j+1)}
\]
for some constant $c = c (l_0)$, where $\Gamma$ is the Euler Gamma function. Since $\Gamma (j + 1-a) \sim \frac{1}{j^{a}} j!$
as $j\to \infty$, we deduce that
$
P_j
\sim
c \frac{1}{j^{a}} 
$,
and so $\sum_j P_j$ converges.
\end{proof}

We can now prove
Theorem \ref{lemma_conv_fai}.
The proof
follows the main ideas of the one of
\cite[Theorem 2.6]{bsu}.
  The major issue  (and the main difference with respect to \cite{bsu})
  will be
  to
  take into account 
  the errors due the $O(y)$-terms in the estimates.
  This will be done by means of the following Lemma, which relies on Propositions \ref{prop_stima_y_prima} and
  \ref{prop_stima_y_centro}.
  
  \begin{lemma}\label{lemma_conv_y}
 Let $p\in \tilde C_0$ and $n_p (\eps)$
 be as in \eqref{eq_def_entering}. Let $\b n (\eps)$ be such that $n_p (\eps) \leq \b n (\eps) \leq \frac{3\pi}{5\aeps}$.
  Then the following hold:
  \begin{enumerate}
  \item\label{item1_lemma_conv_y} the function $\eps \mapsto \sum_{j=1}^{\b n (\eps)}\pa{ \abs{y\pa{ F^j_\eps (p)}} + \abs{y\pa{ F_0^j (p)}}}$
   is bounded, locally uniformly on $p$, for
$\eps$
sufficiently small;
   \item\label{item2_lemma_conv_y}
   $\lim_{\eps \to 0} \sum_{j=n_p (\eps) + 1}^{\b n (\eps)} \abs{y (F^j_\eps (p))} =0$, locally uniformly on $p$.
   \end{enumerate}   
   \end{lemma}

Notice that, by Proposition \ref{prop_stima_lunghezza_centro},
$n'_p (\eps) \geq \frac{\pi- K/2}{\rho'' \aeps} - \frac{M^+}{\rho''}
\geq \frac{7\pi}{8} \frac {5}{4\aeps}  - \frac{M^+}{\rho''}  \geq \frac{3\pi}{5\aeps}$
for $\eps$
sufficiently small.
So, in particular, the orbit up to time $\b n (\eps)$
is contained in $\t C_\eps \cup \t D_\eps$.
On the other hand, we have $n_p (\eps) + \frac{M^-}{2-\rho''} \leq \frac{K}{(2-\rho'')\aeps} 
\leq \frac{\pi/4}{(2-5/4)\aeps} - \frac{M^-}{2-\rho''} \leq \frac{\pi}{3\aeps}$.
So, in particular, the assumption of Lemma
\ref{lemma_conv_y}
is satisfied when $(\eps_\nu, \b n (\eps_\nu))$
is of bounded type.

\begin{proof}
We start with the first point. The convergence of the second part of the series is immediate from
Proposition \ref{prop_hakim_26}, by the harmonic
behaviour of $x(F^j_0 (p))$ and the estimate \eqref{eq_xy_hakim}.
Let us thus consider the first part.
Here we split this series in a first part, with the indices up
to $n_p (\eps)$
and in the remaining part
starting from $n_p (\eps) +1$.
The sum is thus given by
\[
\sum_{j=1}^{n_p (\eps)} \abs{y\pa{ F^j_\eps (p)}} + \sum_{j=n_p (\eps) + 1}^{\b n (\eps)} \abs{y\pa{ F^j_\eps (p)}}
\]
and, by Propositions \ref{prop_stima_y_prima} and \ref{prop_stima_y_centro},
this is bounded by (a constant times) 
\[
\sum_{j=1}^{n_p (\eps)}
 \prod_{l=M^+}^{M^+ + j-1} \pa{1 - \frac{\t \rho}{l}}
+
\pa{\prod_{j=M^+}^{n_p (\eps) -1+ M^+ } \pa{1- \frac{\t \rho}{j}}}
\cdot \sum_{j=n_p (\eps)}^{\b n (\eps)} e^{4\pi\rho\tau}
\]
where $M^+$ is as in \eqref{eq_MM_section_coord}
and $\t \rho$ is (as in Proposition \ref{prop_stima_y_prima}) a constant greater than 1.
By the lower estimates on $n_p (\eps)$ in Proposition \ref{prop_prima_stime_1}
and the asymptotic behaviour proved
in Lemma \ref{lemma_serie}, the last expression is bounded by
\[
\sum_{j=1}^{\infty}  \prod_{l=M^+}^{j-1+M^+} \pa{1 - \frac{\t \rho}{l}}
+
\frac{3	\pi}{5\aeps} \cdot e^{4\pi\rho\tau} \cdot
\pa{\frac{1}{\frac{K}{\rho''\aeps} - \frac{M^+}{\rho''}-1+M^+}}^{\t \rho}.
\]
The first term
is bounded,
again by Lemma \ref{lemma_serie}, and the second one (which, up to a constant, is in particular
a majorant for the sum in the second point in the statement) 
goes to zero  as $\eps \to 0$ (since $\t \rho >1$).
This proves both statements.
\end{proof}

 \begin{proof}[Proof of Theorem \ref{lemma_conv_fai}]
First of all,
recall that
by
Lemma \ref{lemma_prima_coord}
the sequence
$\t \fai_{0,m_\nu}= \tilde w^\iota_0 + \sum_{j=0}^{m_\nu-1} A_0 (F_0^j (p)) $
converges to 
a (1-dimensional) Fatou coordinate $\t \fai$
(for this we just need that $m_\nu \to \infty$).
It is then enough to
show that the difference $\t \fai_{\eps_\nu,m_\nu}- \t \fai_{0,m_\nu}$
goes to zero as $\nu\to\infty$.
Here we shall make use of the hypothesis that the sequence $(\eps_\nu, m_\nu)$ is of bounded type.
The difference is equal to
\[
\t \fai_{\eps_\nu,m_\nu} (p) - \t \fai_{0,m_\nu} (p) = \tilde w^\iota_{\eps_\nu} (p) - \tilde w^\iota_0 (p)+ \sum_{j=0}^{m_\nu-1} \pa{A_{\eps_\nu} (F_{\eps_\nu}^j (p))-A_0 (F^j_0 (p))}
\]
and we see that the first difference goes to zero as $\nu\to \infty$.
We thus only have to estimate the second part, whose modulus
is bounded by
\[
\begin{aligned}
\sum_I + \sum_{II} := 
&
\sum_{j=0}^{m_\nu-1}  \abs{A_0 (F_{\eps_\nu}^j (p))-A_0 (F^j_0 (p))} \\
&
+
 \sum_{j=0}^{m_\nu-1} \abs{A_{\eps_\nu} (F_{\eps_\nu}^j (p))-A_0 (F^j_{\eps_\nu} (p))}.
\end{aligned}
\]
Let us consider the first sum. First of all, we prove that the majorant
$\sum_{j=1}^{m_\nu-1} \pa{  \abs{A_0 (F_{\eps_\nu}^j (p))} + \abs{A_0 (F^j_0 (p))}  }$ converges.
This follows from the fact that $A_0 (p) = O(x^2,y)$
by Proposition \ref{prop_error_A}, the estimates on
$\abs{x (F^j_{0} (p))}$
and
$\abs{x (F^j_{\eps_\nu} (p))}$
in
Propositions
\ref{prop_hakim_26}
and
\ref{prop_stima_prima_x_completa}
and from Lemma \ref{lemma_conv_y} \eqref{item1_lemma_conv_y}.
Indeed, with $M^+$
as in \eqref{eq_MM_section_coord}, we have (for some positive constant $K_0$),
\[
\begin{aligned}
\sum_I & \leq
\sum_{j=1}^{m_\nu-1} \pa{  \abs{A_0 (F_{\eps_\nu}^j (p))} + \abs{A_0 (F^j_0 (p))} } \\
& \leq K_0 \sum_{j=1}^{m_\nu-1} \pa{\abs{x (F^j_{\eps_\nu} (p))}^2 + \abs{x (F_0^j (p))}^2} + K_0 \sum_{j=1}^{m_\nu -1} 
 \pa{\abs{y (F^j_{\eps_\nu} (p))} + \abs{y (F_0^j (p))}}\\
 & \leq K_0 \sum_{j=1}^{m_\nu -1}
  \pa{
 \frac{8}{(j+M^+)^2}
  +\abs{\eps_\nu}^2
  }
 + K_0 \sum_{j=1}^{m_\nu-1} \pa{\abs{y (F^j_{\eps_\nu}(p))} + \abs{y (F_0^j (p))}} \leq B
\end{aligned}
 \]
 where in the last passage we used the assumption that the sequence $(\eps_\nu, m_\nu)$
 is of bounded type to estimate the sum of the $\abs{\eps_\nu}^2$'s and
 in order to apply Lemma \ref{lemma_conv_y} \eqref{item1_lemma_conv_y}
 for the second sum.
 
 We now prove that $\sum_I$ goes to zero, as $\nu \to \infty$.
 Given any small $\eta$, we look for a sufficiently large $J$ such that the sum
\[
\sum_{j=J}^{m_\nu -1} \abs{A_0 (F_{\eps_\nu}^j (p))-A_0 (F^j_0 (p))}
\]
is less than $\eta$ for $\abs{\eps_\nu}$ smaller than some $\eps_0$.
The convergence to 0 of $\sum_I$ will then follow from the fact
that
$A_0 (F_{\eps_\nu}^j (p))- A_0 (F_0^j(p))\to 0$ as $\nu \to \infty$, for every fixed $j$. 
As above, this sum is bounded by
\begin{equation}\label{eq_split_3}
\sum_{j=J}^{m_\nu -1}  \pa{\frac{8}{(j+M^+)^2}+\abs{\eps_\nu}^2} + \sum_{j=J}^{m_\nu -1} \abs{y (F^j_{\eps_\nu} (p))}
+\sum_{j=J}^{m_\nu -1} \abs{y (F_0^j (p))}.
\end{equation}
For $J$ sufficiently large, the first sum is smaller than $\eta/3$ (uniformly in $\eps$), since $(\eps_\nu, m_\nu)$
is of bounded type.
The same is true for the third one, by the harmonic behaviour of $x(F_0^j (p))$ and the estimate \eqref{eq_xy_hakim}.
We are thus left with the second sum of \eqref{eq_split_3}.
We split it as in Lemma \ref{lemma_conv_y}:
\begin{equation}\label{eq_teo_36}
 \sum_{j=J}^{m_\nu -1} \abs{y (F^j_{\eps_\nu} (p))}
 \leq
 \sum_{j=J}^{n_p (\eps_\nu)} \abs{y (F^j_{\eps_\nu} (p))}
+ \sum_{j=n_p (\eps_\nu) +1}^{m_\nu -1} \abs{y (F^j_{\eps_\nu}(p))}.
\end{equation}
Lemma \ref{lemma_conv_y} \eqref{item2_lemma_conv_y}
implies that the second sum of the right hand side
goes to zero as
$\eps_\nu\to 0$.
We are thus left with the first sum in the right hand side of
\eqref{eq_teo_36}.
We estimate it
by applying twice Proposition \ref{prop_stima_y_prima}
and Lemma \ref{lemma_serie}:
\[
\begin{aligned}
\sum_{j=J}^{n_p (\eps_\nu)} \abs{y (F^j_{\eps_\nu} (p))}
& \leq
c_1 \sum_{j=J}^{n_p (\eps_\nu)} \abs{y(F^J_{\eps_\nu} (p))}
\prod_{l=J+M^+}^{j-1+M^+} \pa{1-\frac{\t \rho}{l}
}\\
&
\leq
c_1 \abs{y(F^J_{\eps_\nu} (p))}
\sum_{j=J}^{\infty} 
\prod_{l=J+M^+}^{j-1+M^+} \pa{1-\frac{\t \rho}{l}
}
\\
& \leq
C_1 \abs{y(F^J_{\eps_\nu} (p))}\\
& \leq
C_2 \abs{y(p)}
\prod_{l=M^+}^{J-1+M^+} \pa{1-\frac{\t \rho}{l} 
}.
\end{aligned}\]
We can then take $J$ large enough (and independent from $\eps$) so that
the last term is smaller than
$\frac{\eta}{6}$.
Notice in particular the independence of $J$
from $\eps$ (for $\eps$ sufficiently small).

So,
until now we have proved that $\sum_I$ goes to zero as $\nu\to \infty$.
It is immediate to check that the same holds for
$\sum_{II}$.
Indeed,
\[
\sum_{II}\leq \sum_{j=0}^{m_\nu -1}
\abs{
A_{\eps_\nu} (F^j_{\eps_\nu}(p)) - A_0 (F^j_{\eps_\nu}(p))
}
\leq
\sum_{j=0}^{m_\nu -1}
K_1 \abs{\eps_\nu}^2
\]
for some positive constant $K_1$.
The assertion then follows since $(\eps_\nu, m_\nu)$
is of bounded type.
\end{proof}

\section{The convergence to the Lavaurs map}\label{section_convergence}

In this section we prove Theorem \ref{teo_lav_2d}.
We shall exploit the 1-dimensional Theorem \ref{teo_lavaurs_convergence}, i.e., the convergence of the restriction of
$F_{\eps_\nu}^{n_\nu}$ on $C_0 = \t C_0 \cap \set{y=0}$
to the 1-dimensional Lavaurs map $L_\alpha$.

\begin{lemma}\label{lemma_conv_orizzontale}
Let $p_0\in \t C_0 \cap \set{y=0}$ and $(\eps_\nu, n_\nu)$
 an $\alpha$-sequence. Assume that $q_0 :=L_\alpha (p_0)$
 belongs to $- \t C_0 \cap \set{y=0}$.
Then for every $\delta$ there exists $\eta$ such that
(after possibly shrinking $\t C_0$)
\[
{\t \fao}
\pa{
-\t C_0
\cap
F_{\eps_\nu}^{n_\nu}
\pa{
\t C_0
\cap
\pa{\t \fai}^{-1}
\pa{
\D (
\t \fai (p_0),
\eta
)
}
}
}
\subset
 \D (\t \fao (q_0), \delta)
\]
for every $\nu$
sufficiently large.
\end{lemma}

The need of shrinking $\t C_0$
is just due to the fact that Theorem \ref{lemma_conv_fai} and Corollary
\ref{cor_conv_fao}
give the convergence
on compact subsets of $\t C_0$ (and $-\t C_0$).

\begin{proof}
 Let $m^o_\nu$ and $m^\iota_\nu$
 be sequences of bounded type such that $m^\iota_\nu
 + m^o_\nu = n_\nu$.
By definition of $\t \fai_{\eps, n}$ and $\t \fao_{\eps, n}$
we have
\begin{equation}\label{eq_semiconj}
\begin{aligned}
\t \fao_{\eps_\nu, m^o_\nu} \circ F^{n_\nu}_{\eps_\nu} (p)
& =
 \tilde w_{\eps_\nu}
 \pa{F^{-m^o}_{\eps_\nu}
 (F_{\eps_\nu}^{n_\nu} (p))}
 - \frac{\pi}{2\eps_\nu} + m^o_\nu\\
 & =  \tilde w_{\eps_\nu} \pa{F_{\eps_\nu}^{m^\iota_{\eps_\nu}} (p)} - \frac{\pi}{2\eps_\nu} - m^\iota_\nu + n_\nu\\
 & = \t \fai_{\eps_\nu, m^o_\nu} (p) + n_\nu - \frac{\pi}{\eps_\nu}
\end{aligned}
\end{equation}
whenever $F_{\eps_\nu}^{n_\nu} (p)\in -\t C_0$.
The assertion follows from Theorem \ref{lemma_conv_fai} and Corollary \ref{cor_conv_fao}.
\end{proof}

\begin{lemma}\label{lemma_stima_verticale_globale}
Let $p_0 \in \t C_0
\cap \set{y=0}$ and $(\eps_\nu, n_\nu)$
be a $\alpha$-sequence.
Assume that $q_0 :=L_\alpha (p_0)$
 belongs to $- \t C_0 \cap \set{y=0}$.
Then, for every
polydisc $\Delta_{q_0}$ centered at $q_0$
and contained in $-\t C_0$
there exists a polydisc $\Delta_{p_0}$
centered at $p_0$ and contained in $\t C_0$
such that $F^{n_\nu}_{\eps_\nu} \pa{\Delta_{p_0}} \subset \Delta_{q_0}$
for $\nu$ sufficiently large.
\end{lemma}

\begin{proof}
Set $\Delta_{q_0} = \D_{q_0}^1 \times \D_{q_0}^2$
and analogously
 $\Delta_{p_0} = \D_{p_0}^1 \times \D_{p_0}^2$.
By Lemma \ref{lemma_conv_orizzontale}
it is enough to prove that, if $\Delta_{p_0}$
is sufficiently small, for every $\nu$ sufficiently
large we have
\[
\max_{\D_{p_0}^1 \times \partial \D_{p_0}^2}  \abs{y (F_{\eps_\nu}^{m^\iota_\nu})}
\leq \frac{1}{2} \min_{\D_{q_0}^1 \times \partial \D_{q_0}^2}
\abs{y (F_{\eps_\nu}^{-m^o_\nu})}.
\]
We shall use the estimates collected in Section \ref{section_fasi}.
First of all,
notice that, by Proposition \ref{prop_stima_y_centro}, it is enough to prove that
\[
\max_{p\in \D_{p_0}^1 \times \partial \D_{p_0}^2}
 \abs{y (F_{\eps_\nu}^{n_p (\eps_\nu)})}
\leq c \min_{q \in -\D_{q_0}^1 \times \partial \D_{q_0}^2}
\abs{y (H_{\eps_\nu}^{n_\nu - n'_p (\eps_\nu)})}
\]
for some constant $c$, where $H_\eps$ is as in \eqref{eq_family_he}.
Geometrically,
we want to ensure
that
the vertical expansion
in the third part of the orbit (i.e., after $n'_p (\eps)$)
is balanced by a suitable
contraction
during the first
part (i.e., up to $n_p(\eps)$).

This means proving that
\begin{equation}\label{eq_prodotto_finale}
\begin{aligned}
&\abs{\prod_{j=0}^{n_p (\eps)}  \pa{ 1+\rho x(F_{\eps_\nu}^j (p))
+\beta_{\eps_\nu} (x(F_{\eps_\nu}^j (p)), y(F_{\eps_\nu}^j (p)))
} 
}\\
&\leq
c'
\abs{
\prod_{j=0}^{n_\nu - n'_p (\eps_\nu)}  \pa{ 1+\rho x(H_{\eps_\nu}^j (p)) +\beta^H_{\eps_\nu} (x(H_{\eps_\nu}^j (p)), y(F_{\eps_\nu}^j (p))) }
}
\end{aligned}
\end{equation}
for some positive $c'$.
First of all, we claim that there exists a constant $K_1$ (independent from $\nu$)
such that $K_1 + n_p (\eps_\nu) \geq n_\nu - n'_p (\eps_\nu)$, i.e.,
the number of points in the orbit for $F_\eps$
before entering in $\t D_{\eps_\nu}$ (and thus in the contracting part)
is at least the same (up to the constant) of the number of points in the expanding part.
Indeed,
recalling that
the definition \eqref{eq_def_Ktau} of $K$,
we have
$\frac{K}{\rho'' \abs{\eps_\nu}} \geq \frac{\pi}{\abs{\eps_\nu}} - \frac{\pi-K/2}{\rho'' \abs{\eps_\nu}}$. So, by
Propositions \ref{prop_prima_stime_1} and \ref{prop_stima_lunghezza_centro} we have,
with $M^+$ as in \eqref{eq_MM_section_coord},
\[
\begin{aligned}
1+\abs{\alpha} + \frac{M^+}{\rho''} + n_p (\eps_\nu)&
\geq 1+ \abs{\alpha}  + \frac{M^+}{\rho''} + \frac{K}{\rho'' \abs{\eps_\nu}} - \frac{M^+}{\rho''}\\
&\geq 
n_\nu  - \frac{\pi}{\abs{\eps_\nu}} + \frac{\pi}{\abs{\eps_\nu}} - \frac{\pi - K/2}{\rho'' \abs{\eps_\nu}}
\geq n_\nu - n'_p (\eps_\nu)
\end{aligned}\]
for $\nu$ sufficiently large,
and the desired inequality is proved.
The inequality
\eqref{eq_prodotto_finale}
now follows from Lemma \ref{lemma_confronto_fh} (and Proposition \ref{prop_stima_prima_x_completa}),
and the assertion follows.
\end{proof}

We can now
prove
Theorem \ref{teo_lav_2d}.

\begin{proof}[Proof of Theorem \ref{teo_lav_2d}]
First of all, we can assume that $p_0$ belongs to $C_0 = \set{y=0} \cap \t C_0$. Indeed, there exists some
$N_0$ such that $F^{N_0}_0 (p_0) \in \t C_0$.
So, we can prove the Theorem
for the $(\alpha - N_0)$-sequence $(\eps_\nu, n_\nu - N_0)$
and the base point $F^{N_0}_0 (p_0)$ and the assertion then follows
since
$F_{\eps_\nu}^{N_0} \to F^{N_0}_{0}$.
For the same reason, we can assume that $q_0 := L_\alpha (p_0)$
belongs to $-\t C_0$.

By Lemma \ref{lemma_stima_verticale_globale}, there exists a polydisc $\Delta_{p_0}$
centered at $p_0$
such that the sequence $F^{n_\nu}_{\eps_\nu}$
is bounded on $\Delta_{p_0}$. In particular, up to a subsequence, this sequence converges to a limit map $T_\alpha$,
defined in $\Delta_{p_0}$
with values in $-\t C_0$.
Notice that the limit must be open, since the same arguments apply to the inverse system.
The relation \eqref{eq_teo_lav_2d_semiconj} then follows from \eqref{eq_semiconj} and the assertion follows.
\end{proof}

In the following,
given a subset $\Uu\subset \t C_0$,
we 
denote by
$\Tt_\alpha (\Uu)$ the set
\[
\Tt_\alpha (\Uu)
:=
\set{
T: \Uu \to \C^2 \colon \exists (\eps_\nu,n_\nu) \alpha-\mbox{sequence such that }
F_{\eps_\nu}^{n_\nu} \to T \mbox{ on } \Uu  
}
\]
We denote
by $\Tt_\alpha$ the union of all the $\Tt_\alpha (\Uu)$'s,
where $\Uu \subset \t C_0$,
and call the elements of $\Tt_\alpha$
\emph{Lavaurs maps}.
Theorem \ref{teo_lav_2d} can then be restated as follows: every compact subset $\Cc_0 \subset  C_0$ has a neighbouhhood
$\Uu_{\Cc_0} \subset \t C_0$ such that every
$\Tt_\alpha (\Uu_{\Cc_0})$
is not empty.

\begin{remark}
Computer experiments suggest that
given any $\alpha$-sequence $(\eps_\nu, n_\nu)$
there is a neighbourhood of $C_0$ in $\t C_0$
such that the sequence $F_{\eps_\nu}^{n_\nu}$
converges to a (unique) limit map $T_\alpha$, without the need of extracting a subsequence.
\end{remark}

\section{The discontinuity of the large Julia set}\label{section_disc_J}
In this section we shall prove Theorem \ref{teo_disc_J}. By means of the Lavaurs maps $T_\alpha$,
we first define a 2-dimensional analogous of the Julia-Lavaurs set $J^1(F_0, T_\alpha)$, and use this set
to estimate the discontinuity of the Julia set at $\eps =0$.

\begin{defi}\label{defi_lav_j}
Let $\Uu \subset \t C_0$ and $T_\alpha \in \Tt_\alpha (\Uu)$.
The Julia-Lavaurs set $J^1 (F_0, \alpha)$ is the set
\[
J^1 (F_0, T_\alpha) := \bar{\set{z \in \P^2 \vert \exists m \in \N
\colon
T_\alpha^m (z) \in J^1 (F_0)}}.
\]
\end{defi}

The condition $T_\alpha^m (z) \in J^1 (F_0)$ means that we require $T_\alpha^i (z)$ to be defined, for $i=0,\dots m$.
In particular,
we have
$z, \dots, T_{\alpha}^{m-1} (z) \in \Uu$.

From the definition it follows that
$J^1 (F_0)\subseteq J^1 (F_0, T_\alpha)$, for every $T_\alpha\in \Tt_{\alpha}$.
The following result gives the key estimate for the lower-semicontinuity of the large Julia sets at $\eps=0$.
The proof is analogous to the 1-dimensional case, exploiting the fact that the maps $T_\alpha$ are open.

\begin{teo}\label{teo_liminf_J}
 Let $T_\alpha \in \Tt_\alpha$ be defined on $\Uu\subset \t C_0$
and  $(n_\nu,\eps_\nu)$ be a $\alpha$-sequence
such that $F_{\eps_\nu}^{n_\nu} \to T_\alpha$ on $\Uu$.
Then
 \[
 \liminf J^1 (F_{\eps_\nu}) \supseteq J^1 (F_0, T_\alpha).
 \]
\end{teo}

\begin{proof}
 The key ingredients are the lower semicontinuity of $J^1 (F_\eps)$ and Theorem \ref{teo_lav_2d}.
 By definition, the set of all $z$'s admitting an $m$
 such
 that $T_\alpha^m (z) \in J^1 (F_0)$
 is dense in $J^1 (F_0, T_\alpha)$. Thus, given $z_0$ and $m$
 satisfying the previous condition, we
 only need to find a sequence of points $z_\nu \in J^{1} (F_{\eps_\nu})$
 such that $z_\nu \to z_0$, for some sequence $\eps_\nu \to 0$.
 
 Set $p_0 := T^m_\alpha (z_0)$.
 By the lower semicontinuity
 of $\eps\mapsto J^1 (F_{\eps})$
we can find a sequence of points
 $p_\nu\in J^1 (F_{\eps_\nu})$
such that $p_\nu \to p_0$.
By Theorem \ref{teo_lav_2d} we have $F_{\eps_\nu}^{m n_\nu}\to T^m_{\alpha}$ uniformly near $z_0$,
and this (since $T_\alpha$ is open) gives a sequence $z_\nu$
converging to $z_0$ such that $F_{\eps_\nu}^{mn_\nu} (z_\nu) = p_\nu \in J^1 (F_{\eps_\nu})$.
This implies that $z_\nu \in J^1 (F_{\eps_\nu})$, and
the assertion follows.
\end{proof}

Notice the function $\eps \mapsto J^1(F_{\eps})$
is discontinuous at $\eps=0$ since, by means of just the one-dimensional
Lavaurs Theorem \ref{teo_lavaurs_convergence},
we can create points in $\t C_0 \cap \set{y=0}$
(which is contained in the Fatou set)
satisfying $L_\alpha (p) \in J^1 (F_0)$. 
Indeed, the following property holds:
\begin{equation}\label{eq_property_1dim}
\mbox{$\forall p\in \t C_0 \cap \set{y=0}$
there exists $\alpha$
 such that $p\in J^1 (\pa{F_0}_{|y=0}, L_\alpha)$.}
 \end{equation}
 where $L_\alpha$
 is the 1-dimensional Lavaurs map on the invariant line $\set{y=0}$
 associated to $\alpha$.
 Indeed, since $\partial \Bb \subseteq J^1 (F_0)$ and $\Bb$ intersects the repelling basin $\Rr$,
 we can find $q \in J^1 (F_0) \cap \set{y=0}$
 in the image of the Fatou parametrization $\psi^o$
 for $\pa{F_0}_{|\set{y=0}}$.
 The assertion follows considering $\alpha$ such that $L_\alpha (p)= q$.

In our context, given any $p\in C_0$ and $q \in -C_0$
as above, by means of Theorem \ref{teo_lav_2d}
we can consider a neighbourhood
of $p$ where a sequence $F_{\eps_\nu}^{n_\nu}$
converges to a Lavaurs map $T_\alpha$ (necessarily coinciding with $L_\alpha$
on the line $\set{y=0}$). Since $T_\alpha$ is open, we have that $T_\alpha^{-1} (J^1 (F_0))$
is contained in the liminf of the Julia sets $J^1 (f_{\eps_\nu})$.
This gives a two-dimensional estimate of the discontinuity.

\section{The discontinuity of the filled Julia set}\label{section_altro}

For regular polynomial endomorphism of $\C^2$ it is meaningful to consider the filled Julia set, defined in the following way.

\begin{defi}
Given a regular polynomial endomorphism $F$ of $\C^2$,
the \emph{filled Julia set} $K(F)$ is the set of points whose orbit is bounded.
\end{defi}

Equivalently, 
given any sufficently large ball $B_R$, such that $B_R \Subset F (B_R)$,
the filled Julia set
is equal to
\[
K (F) := \cap_{n\geq 0} F^{-n} (B_R).
\]

In this section we shall prove that, if the family \eqref{eq_family} is induced by regular polynomials,
then the set-valued function $\eps\to K(F_\eps)$ is discontinuos
at $\eps=0$.

Recall that the function $\eps \to K(F_\eps)$ is always upper semicontinuous (see \cite{douady94}).
Here the key definition
will be the following analogous of the filled Lavaurs-Julia set in dimension 1 (\cite{lavaurs}).

\begin{defi}
Given $\Uu \subset \t C_0$ and $T_\alpha \in \Tt_\alpha (\Uu)$, the
 filled Lavaurs-Julia set $K (F_0, T_\alpha)$
 is the complement of the points $p$ such that there exists $m\geq 0$
 such that
 $T_\alpha^m (p)$ is defined and is not in $K (F_0)$.
\end{defi}

Notice in particular that $K(F_0, T_\alpha) \subseteq K(F_0)$ and coincides with $K(F_0)$
outside
$\Uu$. Moreover, notice that $K(F_0, T_\alpha)$
is closed.

\begin{teo}\label{teo_limsup_K}
Let $T_\alpha \in \Tt_\alpha$ be defined on some
$\Uu \subset \t C_0$.
and let 
$(\eps_\nu, n_\nu)$ be an $\alpha$-sequence such that $F^{n_\nu}_{\eps_\nu}\to T_\alpha$
on $\Uu$.
Then
 \[
K (F_0, T_\alpha) \supseteq \limsup
K (F_{\eps_\nu}).
 \]
\end{teo}

\begin{proof}
Since the set-valued function $\eps\mapsto K_\eps$ is upper-semicontinuous,
there exists a large ball $B$ such that, for $\nu \geq \nu_0$, we have $\cup_\nu K (F_{\eps_\nu})\subset B$. Without
loss of generality we can assume that $\nu_0 =1$.
Let us consider the space
\[P:= \set{ \set{0} \cup \bigcup_\nu \{\eps_\nu\}} \times B\]
and its subset $X$
given by
\[X := \set{(0,z) \colon x \in K (F_0, T_\alpha)} \cup  \bigcup_\nu \set{(\eps_\nu, z)\colon z\in K (F_{\eps_\nu})}.\]
By \cite[Proposition 2.1]{douady94} and the fact that $P$ is compact, it is enough to prove that $X$ is closed in $P$.
This follows from Theorem \ref{teo_lav_2d}. Indeed, let $z$ be in the complement of $K(F_0, T_\alpha)$. Since this set is closed,
a small ball $B_z$ around $z$ is outside $K(F_0, T_\alpha)$,
too.
By definition, this means that, for some $m$, we have $T^m (B_z) \subset K(F_0)^c$.
Theorem \ref{teo_lav_2d} implies that, up to shrinking the ball $B_z$, we have
$F_{\eps_\nu}^{n_\nu} (B_z)\subset K(F_0)^c$ for $\nu$ sufficiently
large.
The upper semicontinuity of $\eps\mapsto K (F_\eps)$
then implies that
$F_{\eps_\nu}^{n_\nu} (B_z)\subset K(F_{\eps_\nu})^c$,
for $\nu$ large
enough.
So, $B_z \subset K(F_{\eps_\nu})^c$ and this gives the assertion.
\end{proof}

\begin{cor}
 Let $F_\eps$ be a holomorphic family of
 regular polynomials of $\C^2$
 as in \eqref{eq_family}. Then the set-valued function
$\eps \mapsto K (F_\eps)$ is discontinuous at $\eps=0$.
\end{cor}

\begin{proof}
The argument is the same used to prove the discontinuity
of $J^1 (F_\eps)$
in Section \ref{section_disc_J}.
If the function $\eps \to K(F_\eps)$
were continuous, Theorem \ref{teo_limsup_K} and the fact that $K(F_0,T_\alpha) \subseteq K(F_0)$ for every $\alpha$ would imply that
all the $K(F_0, T_\alpha)$'s were equal to $K(F_0)$. Since $\t C_0 \subseteq K(F_0)$, it is enough to find $p\in \t C_0$
and $\alpha$ such that $p\notin K(F_0, T_\alpha)$.
To do this, it is enough
to take any point $q$ in
$\set{y=0}$
not contained in $K(F_0)$ (recall that $K(F_0)$
is compact) and then consider a point $p\in \t C_0 \cap \set{y=0}$
and  $\alpha$
such that $L_\alpha (p)=q$.
The existence of such points is a consequence of the property \eqref{eq_property_1dim}.
Then, consider a neighbourhood $\Uu$ of $p$ such that some sequence $F^{n_\nu}_{\eps_\nu}$
converges to a Lavaurs map $T_\alpha$ on $\Uu$. The assertion follows since $T_\alpha$ is open and
coincides with $L_\alpha$ on the intersection with the invariant line $\set{y=0}$.
\end{proof}

\bibliography{bib_parab_2d.bib}{}
\bibliographystyle{alpha}

\end{document}